\documentclass[onecolumn]{autart}    

\usepackage{graphicx}           
\usepackage{dcolumn}            
\usepackage{bm}                 

\usepackage{graphics}           
\usepackage{epsfig}             
\usepackage{times}              
\usepackage[fleqn]{amsmath}
\usepackage{mathrsfs}
\usepackage{amssymb}
\usepackage{extarrows}

\usepackage{ntheorem} 

\usepackage{amsfonts}
\usepackage{fancyhdr}
\usepackage{color}
\usepackage{subfigure}
\usepackage{enumerate}
\usepackage{url}

\allowdisplaybreaks[4] 
\newtheorem{theorem}{Theorem}[section] 
\newtheorem{lemma}{Lemma}[section]

\newtheorem{proposition}{Proposition}[section]
\newtheorem{definition}{Definition}
\newtheorem{remark}{Remark}

\newtheorem*{proof}{Proof}

\newcommand{\diff}[1]{\text{d}{#1}}

\begin{document}

\begin{frontmatter}

\title{Relative Stability in the Sup-norm and Input-to-state Stability in the Spatial Sup-norm for Parabolic PDEs}

\author{Jun Zheng$^{1,2}$}\ead{zhengjun2014@aliyun.com},
\author{Guchuan Zhu$^{2}$}\ead{guchuan.zhu@polymtl.ca},
\author{Sergey Dashkovskiy$^{3}$}\ead{sergey.dashkovskiy@mathematik.uni-wuerzburg.de}

\address{$^{1}${School of Mathematics, Southwest Jiaotong University\\
        Chengdu, Sichuan, P. R. of China 611756}\\
        $^{2}$Department of Electrical Engineering, Polytechnique
        Montr\'{e}al\\ P.O. Box 6079, Station Centre-Ville,
        Montreal, QC, Canada H3T 1J4 \\
        $^{3}${Institute of Mathematics, University of W\"{u}rzburg\\
        Emil-Fischer-Str. 40, W\"{u}rzburg,Germany 97074}}


\begin{abstract}
In this paper, we introduce the notion of \emph{relative $\mathcal{K}$-equi-stability} (RKES) to characterize the uniformly continuous dependence of (weak) solutions on external disturbances for nonlinear parabolic PDE systems. Based on the RKES, we prove the input-to-state stability (ISS) in the spatial sup-norm for a class of nonlinear parabolic PDEs with either Dirichlet or Robin boundary disturbances. Two examples, concerned respectively with a super-linear parabolic PDE with Robin boundary condition and a $1$-D parabolic PDE with a destabilizing term, are provided to illustrate the obtained ISS results. Besides, as an application of the notion of RKES, we conduct stability analysis for a class of parabolic PDEs in cascade coupled over the domain or on the boundary of the domain, in the spatial and time sup-norm, and in the spatial sup-norm, respectively. The technique of De Giorgi iteration is extensively used in the proof of the results presented in this paper.
\end{abstract}

\end{frontmatter}
%
\section{Introduction}

{O}{riginally} introduced by  Sontag in {the late} 1980s, the notion {of} \emph{input-to-state stability} (ISS) has been proven to be a convenient tool for describing robust stability of finite dimensional systems with external inputs. The pioneering work on extending the application of ISS to infinite dimensional systems is owe to  \cite{Dashkovskiy:2013} and \cite{Dashkovskiy:2013b}, where different methods were proposed for constructing ISS-Lyapunov functions for abstract equations in Banach spaces, and impulsive systems, respectively. Particularly, as an application of the proposed methods, ISS-Lyapunov functions and ISS estimates were presented for  some nonlinear single and interconnected parabolic PDEs in \cite{Dashkovskiy:2013}. Since then,  the ISS of PDE systems has drawn much attention in the literature of PDE control. {It is} worth noting that applying the classical regularity theory of PDEs to ISS analysis of PDEs having only in-domain disturbances seems to be straightforward, while it is a challenge to establish {the} ISS for PDEs that have external disturbances distributed  on the boundary of the domain.

In recent years, a great effort has been devoted to establishing {the} ISS for PDEs with boundary disturbances; see  \cite{Karafyllis:2018iss,Mironchenko:2019b} for comprehensive surveys on this topic, and \cite{Zheng:2020SCL,Zheng:2020b} for a summary of different approaches for establishing  ISS of PDEs with boundary disturbances. Among the existing literature, the ISS in $L^{1}$-norm and $L^{q }$-norm with ${q}\in [2,+\infty)$ has been well studied for PDEs with boundary disturbances via different methods (see, e.g., \cite{Zheng:2020b}, and  \cite{Jacob:2019,Jacob:2018_SIAM,Karafyllis:2016a,Mironchenko:2019,Mironchenko:2018,Schwenninger:2019,Zheng:201702,Zheng:201804,Zheng:2020c}, etc.), while few results {are concerned with} ISS in $L^{\infty}$-norm except \cite{Karafyllis2017siam}, where {the} ISS in various norms, including weighted $L^{\infty}$-norm, was considered for linear $1$-D PDEs governed by Sturm-Liouville operators by exploiting the eigenfunction expansion and the finite difference scheme; \cite{Karafyllis:2020}, where ISS-style estimates in the spatial sup-norm was established for classical solutions of nonlinear $1$-D parabolic PDEs by using an ISS Lyapunov functional for the sup-norm; and \cite{Zheng:201702}, where under an appropriate boundary feedback law and with compatibility conditions, an ISS estimate in $L^\infty$-norm was established for a $1$-D linear parabolic equation with  {a destabilizing} term.

The aim of this paper is to provide a new {method} for establishing the ISS in the spatial sup-norm for (weak) solutions of a class of higher dimensional nonlinear parabolic PDEs with boundary disturbances, which is different from {those developed} in \cite{Karafyllis2017siam,Karafyllis:2020}, and \cite{Zheng:201702} concerning classical solutions of $1$-D PDEs. More precisely, in order to establish  ISS estimates in the spatial sup-norm, we borrow first the notion {of} \emph{relative stability} (RS) from \cite{Kloeden:1975}, which was used to characterize a kind of relationship of stabilities for two control systems, to describe the uniformly continuous dependence on the external disturbances for weak solutions of nonlinear PDE systems. Then, based on {the} RS in the (spatial and time) sup-norm, we establish {the} ISS in the spatial sup-norm for the considered higher dimensional nonlinear PDEs with Dirichlet and Robin boundary disturbances. {Moreover}, we show how to apply {the property of} RS to {characterize the} stability in the sup-norm, and the spatial sup-norm, respectively, for a class of PDE systems in cascade coupled via the boundary or over the domain.

The main tool exploited in this paper for the proof of various stability {properties is the} De~Giorgi iteration, which {has been} used for the first time to establish ISS estimates for classical solutions of PDEs in \cite{Zheng:201702}. It should be mentioned that in \cite{Zheng:201702} {the} De~Giorgi iteration is used for addressing {the} ISS of $1$-D parabolic PDEs with Dirichlet boundary disturbances by combining the technique of splitting, {which requires certain} compatibility conditions. While in this paper, {the} De~Giorgi iteration is used for not only $1$-D PDEs with Dirichlet boundary conditions, but also {for} higher dimensional PDEs with {either Dirichlet or Robin boundary conditions}. In addition, {unlike} \cite{Zheng:201702}, we {consider in this paper solutions in a weak sense. Therefore, we do not use any splitting technique and impose any compatibility condition, which is an improvement of the results obtained in \cite{Zheng:201702}}.

In {summary}, the main contribution of this paper includes:
\begin{enumerate}[(i)]
\item introducing RKES to describe the uniformly continuous dependence of weak solutions on the external disturbances and establishing RKES estimates for a class of higher dimensional nonlinear parabolic PDEs;

\item establishing {the} ISS in the spatial sup-norm for weak solutions of higher dimensional PDEs with Dirichlet or Robin boundary disturbances by using the property of RKES;


\item  establishing stability estimates in the spatial and time sup-norm and the spatial sup-norm, respectively for a class of {parabolic systems in cascade}, which are interconnected or coupled via the boundary of the domain;

\item extending the usage of De~Giorgi iteration to ISS analysis of higher dimensional PDEs with Robin boundary conditions.
\end{enumerate}


In the rest of the paper, {we introduce first some basic} notations. Section~\ref{Sec. II} presents the problem formulation,  well-posedness, notions on {relative} stability, and the main results on RKES in the sup-norm and ISS in the spatial sup-norm for the considered PDE systems. Section~\ref{Sec. examples} provides two examples to illustrate the obtained ISS results. As an application of REKS presented in Section~\ref{Sec. II}, we show {in Section~\ref{Sec. cascade systems}} how to apply REKS to {obtain} stability estimates in the spatial and time sup-norm and the spatial sup-norm, respectively, for a class of {parabolic systems in cascade} connected over the domain or on the boundary of the domain. Some concluding remarks are given in Section~\ref{Sec. remarks}.

\textbf{Notations.} $\mathbb{R}_+$ denotes the set of positive real numbers and $\mathbb{R}_{\geq 0} := {\{0\}}\cup\mathbb{R}_+$. $\Omega$ denotes a bounded domain in $\mathbb{R}^n(n\geq 1)$ of class $C^2$, that is, $\overline{\Omega}$ is an $n$-dimensional $C^2$-submanifold of $\mathbb{R}^n$ with boundary $\partial \Omega$.  $|\Omega|$ denotes the $n$-dimensional Lebesgue measure of $\Omega$. For any $T>0$, {$Q_T:=\Omega\times (0,T)$}, $\partial_lQ_T:= \partial \Omega \times (0,T)$ {and} $\partial_pQ_T:=\partial \Omega \times (0,T) \cup \{(x,t)|x\in\overline{\Omega},t=0\}$.

$\mathcal {K}:=\{\gamma : \mathbb{R}_{\geq 0} \rightarrow \mathbb{R}_{\geq 0}|\ \gamma(0)=0,\gamma$ is continuous, strictly increasing$\}$,
$ \mathcal {L}:=\{\gamma : \mathbb{R}_{\geq 0}\rightarrow \mathbb{R}_{\geq 0}|\ \gamma$ is continuous, strictly decreasing, $\lim_{s\rightarrow\infty}\gamma(s)=0\}$,
$ \mathcal {K}\mathcal {L}:=\{\beta : \mathbb{R}_{\geq 0}\times \mathbb{R}_{\geq 0}\rightarrow \mathbb{R}_{\geq 0}|\ {\beta (\cdot,t)}\in \mathcal {K}, \forall t \in \mathbb{R}_{\geq 0}$, and $\beta(s,\cdot)\in \mathcal {L}, \forall s \in {\mathbb{R}_{+}}\}$.

Throughout this paper, all   notations on function spaces are standard, which can be found in, e.g., \cite{Evans:2010,Wu2006}.
 Let $ \mathbb{A}:=C^1(\overline{\Omega};\mathbb{R}_+)$, $\mathbb{C}:= C(\overline{\Omega};\mathbb{R}_{\geq 0})$, $\mathbb{M}:= C( \partial{\Omega};\mathbb{R}_+)$,
 $\mathbb{H}:= C^{0,1}((\overline{\Omega}\times \mathbb{R}_{\geq 0})\times \mathbb{R};\mathbb{R})$,
 $\mathbb{F}:= C(\overline{\Omega}\times \mathbb{R}_{\geq 0};\mathbb{R})$, $\mathbb{D}:=C(\partial{\Omega}\times \mathbb{R}_{\geq 0};\mathbb{R})$, {and} $\mathbb{U}:=W^{1,p}(\Omega)$, where $p\geq 2$ and  $p>n$.
For functions $a\in\mathbb{A},c\in\mathbb{C},m\in\mathbb{M}, $ we always denote
\begin{align}\label{underline}
\underline{a}:=\min_{x\in \overline{\Omega}}a>0, \underline{c}:=\min_{x\in \overline{\Omega}}c\geq0,\underline{m}:=\min_{x\in \partial{\Omega}}m>0.
\end{align}
\section{Problem setting and main results}\label{Sec. II}
\subsection{Problem formulation and well-posedness}
For  functions
 $a\in  \mathbb{A},c\in \mathbb{C},m \in  \mathbb{M},h\in \mathbb{H},f\in \mathbb{F},d\in \mathbb{D},u^0 \in \mathbb{U}$,
 we consider the stability of the following higher dimensional nonlinear parabolic system:
\begin{subequations}\label{PDE}
\begin{align}
\mathscr{L}[u]+h(x,t,u)
 =&f~~~~~~~~\text{in}\ \Omega\times \mathbb{R}_+,\\
\mathscr{B}[u]=&d~~~~~~~~\text{on}\ \partial \Omega \times \mathbb{R}_{+},\\
u(\cdot,0)=&u^0(\cdot)~~\text{in}\   \Omega,
\end{align}
\end{subequations}
where $\mathscr{L}[u]:=u_t-\text{div}  \ (a\nabla u) +cu$, and
  \begin{align}\label{Robin}
{\mathscr{B}[u]}:=a\frac{\partial u}{\partial\bm{\nu}}+mu,
\end{align}
or \begin{align}
\mathscr{B}[u]:=u,\label{Dirichlet}
\end{align}
represents the Robin boundary condition, or the Dirichlet boundary condition, respectively.

We always assume that for any $T>0$, there exist a positive constant $c_0$, an increasing function $H:\mathbb{R}_{\geq 0}\rightarrow\mathbb{R}_{\geq 0} $, and a function $\Psi \in C( \mathbb{R}_{\geq 0};\mathbb{R}_{\geq 0})$ satisfying $\Psi(0)=0$, such that
  \begin{subequations}\label{4}
\begin{align}
&|h(x,t,\xi)|\leq c_0(1+|\xi|^{\lambda}),
 |\partial_\xi h(x,t,\xi)|\leq H(|\xi|), \label{4a}\\
&|h(x,s,\xi)-h(x,t,\xi)-f(x,s)+f(x,t)|
\leq  H(|\xi|)\Psi(|s-t|),
\end{align}
\end{subequations}
holds for all $x\in \overline{\Omega},s,t\in[0,T],\xi\in\mathbb{R},$ where $\lambda \in [1,1+\frac{2}{n}]$ is a constant.

We provide  a definition of a (weak) solution of {the} system~\eqref{PDE}.
\begin{definition}\label{weak solution}
\begin{enumerate}
\item[(i)] We say that $u$ is a weak solution of {the} system~\eqref{PDE} with the Robin boundary condition \eqref{Robin}, if for any $T>0$:
 \begin{align*}
 u\in C([0,T];W^{1,p}(\Omega)),
   u(\cdot,0)=u^0(\cdot)\ \text{in}\  \Omega,
\end{align*}
 and the equality
 \begin{align*}
&-\int_{0}^T\int_{\Omega}u\eta_t\diff{x}\diff{t}
 +\int_{0}^T\int_{\Omega} a \nabla u\nabla \eta\diff{x}\diff{t}+\int_{0}^T\int_{\Omega} (cu+h(x,t,u))\eta\diff{x}\diff{t}\notag\\
       =&\int_{0}^T\int_{\Omega} f \eta\diff{x}\diff{t}+\int_{0}^{T}\int_{\partial \Omega}(d-mu)\eta\diff{x}\diff{t}
    +\int_{\Omega} u^0(x)\eta(x,0)\diff{x}
\end{align*}
 holds true for  any  
 $\eta\in  C([0,T];(W^{1,p}(\Omega))')\cap C^1((0,T); L^{p'}(\Omega))$ with $\eta(\cdot,T)=0$ in $\Omega$, where  $ p'=\frac{p}{p-1}$, and  $(W^{1,p}(\Omega))'$ is the dual space
of $W^{1,p}(\Omega)$.
\item[(ii)] We say that $u$ is a weak solution of {the} system~\eqref{PDE} with the Dirichlet boundary condition \eqref{Dirichlet}, if for any $T>0$:
\begin{align*}
 &u\in C([0,T];W^{1,p}(\Omega)),u=d~\text{on}\  \partial_lQ_T,~u(\cdot,0)=u^0(\cdot)\ \text{in}\  \Omega,
\end{align*}
and the equality
\begin{align*}
 -\int_{0}^T\int_{\Omega}u\eta_t\diff{x}\diff{t}
 +\int_{0}^T\int_{\Omega} a \nabla u\nabla \eta\diff{x}\diff{t}+\int_{0}^T\int_{\Omega} (cu+h(x,t,u))\eta\diff{x}\diff{t}
=\int_{0}^T\int_{\Omega} f \eta\diff{x}\diff{t}
  +\int_{\Omega} u^0(x)\eta(x,0)\diff{x}
\end{align*}
holds true for  any  
 $\eta\in  C([0,T];(W^{1,p}_0(\Omega))')\cap C^1((0,T); L^{p'}(\Omega))$ with $\eta(\cdot,T)=0$ in $\Omega$, where $ p'=\frac{p}{p-1}$, and $(W^{1,p}_0(\Omega))'$ is the dual space
of $W^{1,p}_0(\Omega)$.
\end{enumerate}
\end{definition}

For the well-posedness {of the considered problem}, we have the following result.
\begin{proposition}\label{existence}
 System \eqref{PDE} with either the Robin boundary condition \eqref{Robin}, or the Dirichlet boundary condition \eqref{Dirichlet},   admits a unique weak solution belonging to $ C([0,T];W^{1,p}(\Omega))\cap C^1((0,T);L^{p'}(\Omega))$ for any $T>0$, where $ p'=\frac{p}{p-1}$.
\end{proposition}

\begin{proof}
For any $T>0$, by Theorem 14.5 of \cite{Amann:1988}, {the} system~\eqref{PDE} with the Robin boundary condition \eqref{Robin}, or the Dirichlet boundary condition \eqref{Dirichlet},   admits a unique maximal weak solution $u\in C([0,T_0];W^{1,p}(\Omega))\cap C^1((0,T_0);L^{p'}(\Omega))$ with some $T_0\in (0,T]$.
Furthermore, according to Theorem 15.2(i) of \cite{Amann:1988},
if there exists a positive constant $C$ such that the following \emph{a priori} estimate holds true:
 \begin{align}\label{L1}
\|u(\cdot,t)\|_{L^1(\Omega)}\leq C,\forall t\in (0,T_0),
\end{align}
then, such a maximal solution $u$ must exist globally on $[0,T]$. Therefore, it suffices to prove {that} \eqref{L1} holds true for the maximal solution of {the} system~\eqref{PDE} with the Robin boundary condition \eqref{Robin}, or the Dirichlet boundary condition \eqref{Dirichlet},   respectively.

Indeed, for system~\eqref{PDE} with the Robin boundary condition \eqref{Robin}, the estimate in \eqref{L1} is  guaranteed by \cite[Theorem 3.1]{Zheng:2020b}. For system~\eqref{PDE} with the Dirichlet boundary condition \eqref{Dirichlet}, the estimate in \eqref{L1} is guaranteed by applying the estimate given in Theorem~\ref{Theorem ISS}(ii) to \eqref{PDE} defined over $Q_{T_0}$.
$\hfill\blacksquare$\end{proof}

\begin{remark}
The growth conditions on the nonlinear term $h$ {appearing} in \eqref{4} are only used  for guaranteeing the existence and uniqueness of a weak solution. In particular, as indicated in \cite{Amann:1988}, the exponent $1+\frac{2}{n}$ in \eqref{4a} is optimal for the existence of a global weak solution. However, if a certain compatibility condition is imposed, and a smooth solution is considered,  then the growth conditions on the nonlinear term $h$ in \eqref{4} can be relaxed; see \cite[Proposition 2.1]{Zheng:2020b}.
\end{remark}

\subsection{Notion on {relative} stability}
From the point of view on PDEs, both the initial value and enforced terms (external disturbances) have a deep effect  on the stability (and well-posedness) of PDE  systems. In order to describe the influence {induced} by these data,  we define some stability {characteristics} for PDE systems. More precisely, based on the notion {of} \emph{relative stability} (RS) given by \cite{Kloeden:1975}, which is an extension of the concept for finite dimensional systems, e.g. \cite{Lakshmikantham:1962}, to general control systems, we define several {properties of} relative stability for the considered PDE systems.

 {To emphasize the dependence of the solutions on the initial value and external disturbances}, we denote by $  \Sigma (\mathbb{U},\mathbb{F},\mathbb{D})$ the system~\eqref{PDE} with data $(u^0,f,d)\in \mathbb{U}\times \mathbb{F}\times\mathbb{D}$. Note that if $v \in W^{1,p}(\Omega)~\text{with}~p>n$, then $v\in C^{1-\frac{n}{p}}(\overline{\Omega})$; see, e.g. \cite[Theorem 1.3.2]{Wu2006}. Thus, $\sup_{x\in \Omega}|v(x)|<+\infty$ and hence, $\sup_{x\in \Omega}|v(x)|$ is well-defined.

\begin{definition}\label{UC}
 {The  system $\Sigma (\mathbb{U},\mathbb{F},\mathbb{D})$  is said to be} relatively equi-stable  (RES) in the sup-norm with respect to (w.r.t.) in-domain and boundary disturbances in $\mathbb{F}\times \mathbb{D}$,
if for every constant $\varepsilon>0$, there exists a positive constant $\delta$ depending only on $\varepsilon$, such that  the following implication
\begin{align*}
\sup_{(x,t)\in Q_T}|f_1(x,t)-f_2(x,t)|+ \sup_{(x,t)\in\partial_lQ_T}|d_1(x,t)-d_2(x,t)|  <\delta
     \Rightarrow
     \sup_{(x,t)\in Q_T}|u_1(x,t)-u_2(x,t)|<\varepsilon
\end{align*}
holds true for all $ (f_1,d_1),(f_2,d_2)\in \mathbb{F}\times \mathbb{D}$  and all $T>0$, where $u_i$ is the solution of {the} system $\Sigma (\mathbb{U},\mathbb{F},\mathbb{D})$ corresponding to the data $ (u^0,f_i,d_i)\in \mathbb{U}\times\mathbb{F}\times\mathbb{D}, i=1,2 $.
\end{definition}
\begin{definition}\label{UKC}
 {The  system $\Sigma (\mathbb{U},\mathbb{F},\mathbb{D})$  is said to be} relatively  $\mathcal{K}$-equi-stable (RKES) in the sup-norm w.r.t. in-domain and boundary disturbances in $\mathbb{F}\times \mathbb{D}$, if there exist functions $\gamma_d,\gamma_f\in \mathcal {K}$ such that
\begin{align*}
      \sup_{(x,t)\in Q_T}|u_1(x,t)-u_2(x,t)|
     \leq \gamma_f \left(\sup_{(x,t)\in Q_T}|f_1(x,t)-f_2(x,t)|\right)
      +\gamma_d \left(\sup_{(x,t)\in\partial_lQ_T}|d_1(x,t)-d_2(x,t)|\right)
      ,\forall T>0,
\end{align*}
where $u_i$ is the solution of {the} system $\Sigma (\mathbb{U},\mathbb{F},\mathbb{D})$  corresponding to the data $ (u^0,f_i,d_i)\in \mathbb{U}\times\mathbb{F}\times\mathbb{D}, i=1,2$.

Particularly,  {the  system $\Sigma (\mathbb{U},\mathbb{F},\mathbb{D})$  is said to be} relatively  Lipschitz-equi-stable (RLES)  in the  sup-norm w.r.t. in-domain and boundary disturbances in $\mathbb{F}\times \mathbb{D}$, if $\gamma_f(s)=L_fs,\gamma_d(s)=L_ds$ for any $s\geq0$, where $L_f$ and $L_d$ are certain positive constants.
\end{definition}
\begin{remark} We provide some comments on RES.
\begin{enumerate}[(i)]
\item It should be noticed that there is a slight difference for RS between the definition given in \cite{Kloeden:1975} and the one given in this paper. Indeed, RES defined in \cite{Kloeden:1975} is mainly used to describe the relationship of stabilities between two  systems, namely it is in term of multiple different systems, while  RES defined in this paper is mainly used to {characterize} the uniformly continuous dependence of the solution on {the} external disturbances for a certain PDE, namely it is in term of one system (i.e., the {nominal dynamics} of the PDEs are governed by the same differential operator) with different external inputs.
\item Note that a system with external disturbances may {be RES w.r.t. external disturbances while not being} asymptotically stable. For example, we consider the following systems:
\begin{align*}
&(u_k)_t-(u_k)_{xx}=f_k(x,t),~~~~~~~~~~(x,t)\in \left(0,\frac{\pi}{2}\right) \times  \mathbb{R}_{\geq 0},\\
&u_k(0,t)= 0,
u_k\left(\frac{\pi}{2},t\right)=d_k(t),~  t\in \mathbb{R}_{+},\\
&u_k(x,0)=u^0(x),~~~~~~~~~~~~~~~~~~~~~~~~x\in \left(0,\frac{\pi}{2}\right),
      \end{align*}
         where $u^0(x):=0$,  $f_k(x,t):=\sqrt{2}k\sin x\cos\left(t-\frac{\pi}{4}\right)$, $d_k(t):=k\sin t $, and $k\in\mathbb{R}_{+}$.
{It is} clear that $u_k(x,t)=k\sin t \sin x$ is the unique solution, which is bounded  for any fixed $k$. Since there exists a {point $(x_0,t_0)$} such that $u_k(x_0,t_0)>\frac{k}{2}$, $u_k$ is unbounded as $k\rightarrow +\infty$. However, noting that
                           \begin{align*}
\sup_{(x,t)\in \left(0,\frac{\pi}{2}\right)\times (0,T)}|f_k(x,t)-f_l(x,t)|
=\sqrt{2}|k-l|\sup_{(x,t)\in \left(0,\frac{\pi}{2}\right)\times (0,T)}|\sin x\cos\left(t-\frac{\pi}{4}\right)|,\forall T>0,
     \end{align*}
     and
\begin{align*}
&\sup_{t\in (0,T)}|d_k(t)-d_l(t)|=|k-l| \sup_{t\in (0,T)}|\sin t|,\forall T>0,
      \end{align*}
       for all $ k,l\in \mathbb{R}_{+}$, then the  system is RLES, having the estimate for all $k,l\in\mathbb{R}_{+}$:
\begin{align*}
\sup_{(x,t)\in \left(0,\frac{\pi}{2}\right)\times (0,T)}|u_k(x,t)-u_l(x,t)|
=& |k-l|\sup_{(x,t)\in \left(0,\frac{\pi}{2}\right)\times (0,T)}|\sin t \sin x|\\
\leq  &\sup_{(x,t)\in \left(0,\frac{\pi}{2}\right)\times (0,T)}|f_k(x,t)-f_l(x,t)|+ \sup_{t\in (0,T)}|d_k(t)-d_l(t)|,\forall T>0.
      \end{align*}
      \item                           It is obvious that RLES $\Rightarrow$ RKES $\Rightarrow$ RES.
   \end{enumerate}
\end{remark}
\begin{definition}[\cite{Mironchenko:2018}]\label{0-KLS}
 {The  system $\Sigma (\mathbb{U},\mathbb{F},\mathbb{D})$  is said to be} globally asymptotically stable at zero    uniformly w.r.t. the state (0-UGAS w.r.t. the state)  in the spatial sup-norm
if there exists a function $ \beta\in \mathcal {K}\mathcal {L}$ such that
\begin{align*}
    \sup_{x\in \Omega}|u(x,T)|\leq &\beta\left( \sup_{x\in \Omega}|u^0(x)|,T\right),\forall u^0\in \mathbb{U}, \forall T>0,
\end{align*}
where $ u$ is the solution of the system $\Sigma (\mathbb{U},\mathbb{F},\mathbb{D})$ corresponding to  the data $ (u^0,0,0)\in \mathbb{U}\times\mathbb{F}\times\mathbb{D}$.
\end{definition}
\begin{definition}\label{definition 1}
 {The}  system $\Sigma (\mathbb{U},\mathbb{F},\mathbb{D})$  is said to be input-to-state stable (ISS) in the spatial sup-norm
 w.r.t.  in-domain and boundary disturbances in $\mathbb{F}\times \mathbb{D}$, if there exist functions $\beta\in \mathcal {K}\mathcal {L}$ and $ \gamma_f,\gamma_d\in \mathcal {K}$ such that for all $(u^0,f,d)\in \mathbb{U}\times\mathbb{F}\times\mathbb{D}$:
\begin{align}\label{Eq: ISS def}
      \sup_{x\in \Omega}|u(x,T)|
     \leq &\beta\left( \sup_{x\in \Omega}|u^0(x)|,T\right)
      +\gamma_f \left(\sup_{(x,t)\in Q_T}|f (x,t)|\right)+\gamma_d \left(\sup_{(x,t)\in\partial_l Q_T}|d(x,t)|\right),\forall T>0,
\end{align}
where $u$ is the solution of the system $\Sigma (\mathbb{U},\mathbb{F},\mathbb{D})$ corresponding to the data $(u^0,f,d)\in \mathbb{U}\times\mathbb{F}\times\mathbb{D}$.

Furthermore, the system $\Sigma (\mathbb{U},\mathbb{F},\mathbb{D})$ is said to be exponentially input-to-state stable (EISS) in the spatial sup-norm
 w.r.t. in-domain and boundary disturbances in $\mathbb{F}\times \mathbb{D}$, if there exist constants $M,{\sigma} > 0$ such that $\beta( r,t) =Mre^{-{\sigma} t}$ in \eqref{Eq: ISS def} for all $r\geq 0$.
\end{definition}
\begin{proposition} \label{Proposition UKC}
If  {the  system} $\Sigma (\mathbb{U},\mathbb{F},\mathbb{D})$ is RKES in the  sup-norm w.r.t. in-domain and boundary disturbances in $\mathbb{F}\times \mathbb{D}$ and 0-UGAS w.r.t. the state  in the spatial sup-norm, then {it} is ISS in the spatial sup-norm  w.r.t. in-domain and boundary disturbances in $\mathbb{F}\times \mathbb{D}$.
 \end{proposition}
 \begin{proof}
For any $(u^0,f,d)\in\mathbb{U}\times \mathbb{F}\times \mathbb{D}$, let $u,v$ be the solutions of {the} system $\Sigma (\mathbb{U},\mathbb{F},\mathbb{D})$ with data $(u^0,f,d)$ and $(u^0,0,0)$, respectively. For simplicity, we write $|g|_{\infty,\omega}:=\sup_{y\in\omega}|g(y)|$ for a function $g$ defined on a domain $\omega$ of $\mathbb{R}^{n}$  or $\mathbb{R}^{n+1}$. Since $ \Sigma (\mathbb{U},\mathbb{F},\mathbb{D})$ is RKES in the sup-norm w.r.t. in-domain and boundary disturbances, and $u$ and $v$ are continuous in $T$ (see {Proposition~\ref{existence}}), there exist functions $\gamma_d,\gamma_f\in \mathcal {K}$ such that
\begin{align*}
   |u(\cdot,T)-v(\cdot,T)|_{\infty,\Omega}
    \leq&  |u-v|_{\infty,Q_T}\\
    \leq &\gamma_f \big(|f-0|_{{\infty}, Q_T}\big)+\gamma_d \big(|d-0|_{{\infty},\partial_lQ_T}\big)
     \\
      =&\gamma_f \big(|f|_{{\infty}, Q_T}\big)+\gamma_d \big(|d|_{\infty,\partial_lQ_T}\big), \ \forall T>0.
      \end{align*}
 Since $ \Sigma (\mathbb{U},\mathbb{F},\mathbb{D})$ is 0-UGAS w.r.t. the state in the spatial sup-norm, there exists a function $ \beta\in \mathcal {K}\mathcal {L}$ such that
 \begin{align*}
   |v(\cdot,T)|_{\infty,\Omega}\leq &\beta( |u^0|_{\infty,\Omega)},T) \ \forall T>0.
\end{align*}
It follows that
 \begin{align*}
   |u(\cdot,T)|_{\infty,\Omega} \leq& |u(\cdot,T)-v(\cdot,T)|_{\infty,\Omega}+ |v(\cdot,T)|_{\infty,\Omega}
   \leq \beta( |u^0|_{\infty,\Omega},T)
      +\gamma_f \big(|f|_{{\infty}, Q_T}\big)
      +\gamma_d \big(|d|_{{\infty},\partial_lQ_T}\big), \ \forall T>0,
\end{align*}
which implies that $\Sigma (\mathbb{U},\mathbb{F},\mathbb{D})$ is ISS in the spatial sup-norm
 w.r.t.  in-domain and boundary disturbances in $\mathbb{F}\times \mathbb{D}$.
$\hfill\blacksquare$\end{proof}
\begin{remark}
In general, for a nonlinear system, it is not an easy task to establish the ISS w.r.t. boundary disturbances directly.
{Proposition~\ref{Proposition UKC} provides an alternative for ISS analysis of 0-UGAS systems, which amounts to only assessing the properties of RKES and may be more easily obtained.}
\end{remark}

\subsection{Main stability results}\label{Sec: main results}
Assume further that
 \begin{align}\label{increasing property}
 (h(x,t,\xi_1)-h(x,t,\xi_2))(\xi_1-\xi_2)\geq 0
 \end{align}
 for all $x\in \overline{\Omega},t\in \mathbb{R}_{\geq 0},\xi_1,\xi_2\in\mathbb{R}$.

The first main result is on the RKES in the  sup-norm w.r.t. in-domain and boundary disturbances, whose proof is  provided in Appendix.
\begin{theorem} \label{Theorem RKES-Robin} The following statements hold true.
 \begin{enumerate}
  \item[(i)] Assume that $\underline{c}>0$ {in \eqref{underline}}. System \eqref{PDE} with the Robin boundary condition \eqref{Robin} is RLES in the  sup-norm w.r.t. in-domain and boundary  disturbances in $\mathbb{F}\times \mathbb{D}$, having the estimate:
\begin{align}\label{190625}
\sup_{(x,t)\in Q_T}|u_1(x,t)-u_2(x,t)|
 \leq& \frac{2C_{S}^2}{\min\{\underline{a},\underline{c}\}}|\Omega|^{\frac{q-2}{q}}2^{\frac{3q-4}{2q-4}} \sup_{(x,t)\in Q_T}|f_1(x,t)-f_2(x,t)|\notag\\
     &+\frac{1}{\underline{m}}\sup_{(x,t)\in\partial_l Q_T}|d_1(x,t)-d_2(x,t)| ,\forall T>0,
\end{align}
for all $ (f_i,d_i)\in \mathbb{F}\times\mathbb{D},i=1,2$, where $u_i $ is the solution of the system corresponding to the data $ (u^0,f_i,d_i)\in \mathbb{U}\times\mathbb{F}\times\mathbb{D}, i=1,2 $, and {$q$ and $C_S$ are constants specified in Lemma~\ref{Sobolev inequality}(i)}.
\item[(ii)] Assume that $\underline{c}\geq0$ {in \eqref{underline}}. System \eqref{PDE} with the Dirichlet boundary condition \eqref{Dirichlet} is RLES in the  sup-norm w.r.t. in-domain and boundary  disturbances in $\mathbb{F}\times \mathbb{D}$, having the estimate:
\begin{align}
 \sup_{(x,t)\in Q_T}|u_1(x,t)-u_2(x,t)|
 \leq& \frac{C_{P}^2}{\underline{a}}|\Omega|^{\frac{q-2}{q}}2^{\frac{3q-4}{2q-4}} \sup_{(x,t)\in Q_T}|f_1(x,t)-f_2(x,t)|\notag\\
     +&\sup_{(x,t)\in\partial_l Q_T}|d_1(x,t)-d_2(x,t)|,\forall T>0,\label{012801}
\end{align}
for all $ (f_i,d_i)\in \mathbb{F}\times\mathbb{D},i=1,2$, where $u_i $ is the solution of the system corresponding to the data $ (u^0,f_i,d_i)\in \mathbb{U}\times\mathbb{F}\times\mathbb{D}, i=1,2 $, and {$q$ and $C_P$ are constants specified in Lemma~\ref{Sobolev inequality}(ii)}.

 {Furthermore, if $\underline{c}>0$, it holds that
\begin{align}
 \sup_{(x,t)\in Q_T}|u_1(x,t)-u_2(x,t)|
 \leq&  C_0|\Omega|^{\frac{q-2}{q}}2^{\frac{3q-4}{2q-4}} \sup_{(x,t)\in Q_T}|f_1(x,t)-f_2(x,t)|\notag\\
     +&\sup_{(x,t)\in\partial_l Q_T}|d_1(x,t)-d_2(x,t)|,\forall T>0,\label{012802}
\end{align}
for all $ (f_i,d_i)\in \mathbb{F}\times\mathbb{D},i=1,2$, where $u_i $ is the solution of the system corresponding to the data $ (u^0,f_i,d_i)\in \mathbb{U}\times\mathbb{F}\times\mathbb{D}, i=1,2 $, $C_0:=  \min\left\{\frac{2C_{S}^2}{\min\{\underline{a},\underline{c}\}},\frac{C_{P}^2}{\underline{a}}\right\}$, and $q,C_S,C_P$ are constants specified in Lemma~\ref{Sobolev inequality}.}
\end{enumerate}
\end{theorem}

 As an application of RKES, we have the second main result on the ISS in the spatial sup-norm  for system~\eqref{PDE} with  in-domain and boundary disturbances, whose proof is also provided in Appendix.


\begin{theorem} \label{Theorem ISS} The following statements hold true.
\begin{enumerate}
\item[(i)] Assume that $\underline{c}>0$  {in \eqref{underline}}. System \eqref{PDE} with the Robin boundary condition \eqref{Robin} is EISS in the spatial sup-norm
 w.r.t. in-domain and boundary disturbances in $\mathbb{F}\times \mathbb{D}$, having the  estimate:
\begin{align}\label{++13''}
\sup_{x\in \Omega}|u(x,T)|
 \leq &\sup_{x\in\Omega}|u^0(x)|e^{-\underline{c}  T}
      +\frac{2C_{S}^2}{\min\{\underline{a},\underline{c}\}}|\Omega|^{\frac{q-2}{q}}2^{\frac{3q-4}{2q-4}}  \sup_{(x,t)\in Q_T}|f(x,t)| +\frac{1}{\underline{m}}\sup_{(x,t)\in\partial_l Q_T}|d(x,t)|,\forall T>0,
\end{align}
where {$q$ and $C_S$ are constants specified in Lemma~\ref{Sobolev inequality}(i)}.
\item[(ii)] Assume that $\underline{c}>0$  {in \eqref{underline}}. System \eqref{PDE} with the Dirichlet boundary condition \eqref{Dirichlet} is EISS in the spatial sup-norm
 w.r.t. in-domain and boundary disturbances in $\mathbb{F}\times \mathbb{D}$, having the  estimate:
\begin{align}\label{++13'''}
\sup_{x\in \Omega}|u(x,T)|
 \leq  \sup_{x\in\Omega}|u^0(x)|e^{-\underline{c} T}
      +{C_0}|\Omega|^{\frac{q-2}{q}}2^{\frac{3q-4}{2q-4}} \sup_{(x,t)\in Q_T}|f(x,t)| +\sup_{(x,t)\in\partial_l Q_T}|d(x,t)|,\forall T>0,
\end{align}
where {the constants $q,C_0$ are  the same as in \eqref{012802}}.
\end{enumerate}
\end{theorem}
\begin{remark}
{For Theorem~\ref{Theorem ISS}, it is possible to weaken the condition  $ (u^0,f,d)\in \mathbb{U}\times\mathbb{F}\times \mathbb{D}$ to $(u^0,f,d)\in L^{\infty}(\Omega) \times L^{\infty}_{loc}(\mathbb{R}_{\geq 0};L^{\infty}(\Omega))\times L^{\infty}_{loc}(\mathbb{R}_{\geq 0};L^{\infty}(\Omega))$, and hence obtain
\begin{align*}
\|u(\cdot,T)\|_{L^\infty(\Omega)}
 \leq &  \|u^0\|_{L^\infty(\Omega)}e^{-\underline{c}  T}
      +\frac{2C_{S}^2}{\min\{\underline{a},\underline{c}\}}|\Omega|^{\frac{q-2}{q}}2^{\frac{3q-4}{2q-4}} \|f\|_{L^\infty(Q_T)} +\frac{1}{\underline{m}}\|d\|_{L^\infty(Q_T)},\forall T>0,
\end{align*}
and
\begin{align*}
\|u(\cdot,T)\|_{L^\infty(\Omega)}
 \leq & \|u^0\|_{L^\infty(\Omega)}e^{-\underline{c} T}
       +{C_0}|\Omega|^{\frac{q-2}{q}}2^{\frac{3q-4}{2q-4}} \|f\|_{L^\infty(Q_T)} +\|d\|_{L^\infty(Q_T)},\forall T>0,
\end{align*}
for a weak solution (in a certain sense differing from Definition~\ref{weak solution}) of the system \eqref{PDE} with the Robin boundary condition \eqref{Robin}, and the Dirichlet boundary condition \eqref{Dirichlet}, respectively. Indeed, letting $\{u^0_n\},\{f_n\},\{d_n\}$ be sequences of sufficiently smooth functions, which satisfy $(u^0_n,f_n,d_n)\rightarrow (u^0,f,d)$ in $ L^{\infty}(\Omega) \times L^{\infty}(Q_T)\times L^{\infty}(Q_T)$  as $n\rightarrow +\infty$, we consider the approximating equation  \eqref{PDE} with  data $(u^0_n,f_n,d_n)$, and establish uniform \emph{a priori} estimates of  strong (or smooth)  solutions $\{u_n\}$ as in \cite{Zheng:2020TCL}. For the existence of a weak solution $u$, we may prove by using the uniform \emph{a priori} estimates of  $\{u_n\}$ and taking limits in  appropriate functional spaces. For the ISS in the spatial $L^{\infty}$-norm of $u_n$, we may prove as in the proof of Theorem~\ref{Theorem ISS}; see Appendix. Then, by taking limits of $(u_n,u^0_n,f_n,d_n)$ within the ISS estimates of $\{u_n\}$, we may obtain the  aforementioned ISS estimates of $u$.}
\end{remark}

\section{Illustrative examples on ISS in the spatial sup-norm {of parabolic PEDs}}\label{Sec. examples}
In this section, we illustrate the results on ISS estimates in the spatial sup-norm presented in Section~\ref{Sec. II} {through two examples}.
\subsection{A super-linear parabolic equation with a Robin boundary condition}
Consider the following super-linear parabolic equation:
\begin{subequations}\label{PDE-super-linear}
\begin{align}
u_t-\Delta u +cu+ u\ln (1+u^2)
 =&f~~~~~~~~\text{in}\ \Omega\times \mathbb{R}_+,\\
\frac{\partial u}{\partial\bm{\nu}}+mu=&d~~~~~~~~\text{on}\ \partial \Omega \times \mathbb{R}_{+},\\
u(\cdot,0)=&u^0(\cdot)~~\text{in}\   \Omega,
\end{align}
\end{subequations}
where $\Omega$ is an open bounded domain in $\mathbb{R}^n (n\geq 3)$ with a smooth boundary $\partial\Omega$, $\Delta$ is the Laplace operator, $c,m$ are positive constants with $c\geq 1$, $u^0\in\mathbb{U}$,  $d\in \mathbb{D} $, and  $f\in \mathbb{F}$.  We have the following result.

\begin{proposition}Assume that there exists a point $x^0\in\partial \Omega$ such that $\partial \Omega$ is flat near $x^0$ and lies in the plane $\{x:=(x_1,x_2,...,x_n)\in \mathbb{R}^n|~x_i=0\}$ for some $i\in\{1,2,...,n\}$, and $f(x,t)$ is  Lipschitz continuous in {$t\in \mathbb{R}_{\geq 0}$} for all $x\in \overline{\Omega}$. Then, {the} system~\eqref{PDE-super-linear} is EISS in the spatial sup-norm, having the  estimate:
\begin{align*}
\sup_{x\in \Omega}|u(x,T)|
 \leq &\sup_{x\in\Omega}|u^0(x)|e^{-c  T}
      +  \frac{9(n-1)^2}{(n-2)^2}   |\Omega|^{\frac{2}{n} }2^{\frac{n}{4}+8} \sup_{(x,t)\in Q_T} |f(x,t)| +\frac{1}{m}\sup_{(x,t)\in\partial_l Q_T}|d(x,t)|, \ \forall T>0.
\end{align*}
\end{proposition}
\begin{proof}
Setting $h(x,t,u):= u\ln (1+u^2)$, it is easy to verify that $h$ and $f$ satisfy the conditions proposed in Section~\ref{Sec. II}. Then according to Theorem~\ref{Theorem ISS}(i), system \eqref{PDE-super-linear} is EISS in the spatial sup-norm, having the estimate:
\begin{align}
\sup_{x\in \Omega}|u(x,T)|
 \leq &\sup_{x\in\Omega}|u^0(x)|e^{-c  T}
      +  C_{S}^2  |\Omega|^{\frac{q-2}{q}}2^{\frac{5q-8}{2q-4}} \sup_{(x,t)\in Q_T} |f(x,t)| +\frac{1}{m}\sup_{(x,t)\in\partial_l Q_T}|d(x,t)|,\forall T>0,\label{010902}
\end{align}
where {$q$ and $C_S$ are constants specified in Lemma~\ref{Sobolev inequality}(i)}.

Note that $\partial \Omega$ is flat near $x^0$ and lies in the   plane $\{x \in \mathbb{R}^n|~x_i=0\}$, then in  the following  Sobolev inequality \begin{align}
\|v\|_{L^{2^*}(\Omega)}\leq C_1(\|v\|_{L^2(\Omega)}+\|\nabla v\|_{L^2(\Omega)}),\forall v\in W^{1,2}(\Omega),\label{01091}
\end{align}
 the constant $C_1$ can be chosen as $C_1:=\frac{2(n-1)}{n-2}\times(1+3+4\times 2)=\frac{24(n-1)}{n-2}$;  see \cite[Theorem 2, \S 5.6.1]{Evans:2010}, whose proof  is based on Step 2 of the proof of  \cite[Theorem 1, \S 5.6.1]{Evans:2010}, and Steps 1-4 of the proof of \cite[Theorem 1, \S 5.4]{Evans:2010}.

 For any $q\in (2,2^*)$ with $2^*:=\frac{2n}{n-2}$, it follows from the H\"{o}lder's inequality  that
 \begin{align*}
\|v\|_{L^{q }(\Omega)}^{q }\leq  \|v\|_{L^{2^*}(\Omega)} ^{q} \|1\|_{L^{\frac{2^*}{2^*-q}}(\Omega)} =\|v\|_{L^{2^*}(\Omega)} ^{q}  |\Omega|^{\frac{ 2^*-q}{2^*}},
\end{align*}
 which along with \eqref{01091} gives for  $v\in W^{1,2}(\Omega)$:
 \begin{align*}
\|v\|_{L^{q}(\Omega)}\leq \frac{24(n-1)}{n-2} |\Omega|^{\frac{ 2^*-q}{2^*q}} (\|v\|_{L^2(\Omega)}+\|\nabla v\|_{L^2(\Omega)}).
\end{align*}
Thus $C_S$ in Lemma~\ref{Sobolev inequality}(i) can be chosen as $C_S:=\frac{24(n-1)}{n-2} |\Omega|^{\frac{ 2^*-q}{2^*q}}$. Finally, \eqref{010902} becomes
\begin{align*}
 \sup_{x\in \Omega}|u(x,T)|
 \leq &\sup_{x\in\Omega}|u^0(x)|e^{-c  T}
      +  \frac{24^2(n-1)^2}{(n-2)^2}   |\Omega|^{\frac{q-2}{q}+\left(\frac{ 2^*-q}{2^*q}\right)^2}2^{\frac{5q-8}{2q-4}} \sup_{(x,t)\in Q_T} |f(x,t)|\notag\\
     &+\frac{1}{m}\sup_{(x,t)\in\partial_l Q_T}|d(x,t)|, \ \forall T>0.
\end{align*}
Letting $q\rightarrow 2^*$ and by algebraic computations, we obtain the desired result.
$\hfill\blacksquare$\end{proof}
\begin{remark}
 Note that $\frac{24(n-1)}{n-2} |\Omega|^{\frac{ 2^*-q}{2^*q}}$ is not the best embedding constant in the Sobolev inequality given in Lemma~\ref{Sobolev inequality}(i).
 \end{remark}

\subsection{A $1$-D parabolic PDE  with a destabilizing term}
It is well-known that the following parabolic system
\begin{subequations}\label{unstable PDE0}
\begin{align}
u_t(x,t)=&u_{xx}(x,t)+cu(x,t),~(x,t)\in (0,1)\times\mathbb{R}_+,\\
u(0,t)=&u(1,t)=0,~~~~~~~~~~~~~~~t\in\mathbb{R}_+,\\
u(x,0)=&u^0(x),~~~~~~~~~~~~~~~~~~~~~~~x\in (0,1).
\end{align}
\end{subequations}
is unstable if the constant $c$ is positive and large.
{While, the system~\eqref{unstable PDE0} can be stabilized for any $c\in \mathbb{R}_{+}$ by a backstepping boundary control $u(1,t) = U(t)$ of the form \cite{Liu:2003,Krstic:2008}
\begin{align}\label{+24}
U(t):=-\int_0^1 k(1,y)u(y,t)\text{d}y,
\end{align}
where $k$ is the solution of the following equation
\begin{subequations}\label{PDE-k}
 \begin{align}
& k_{xx}(x,y)-k_{yy}(x,y)= (c +{\sigma})k(x,y),
 ~~~~~~~~0\leq y\leq x\leq 1,\\
&k(x,0)= 0,~~~~~~~~~~~~~~~~~~~~~~~~~~~~~~~~~~~~~~~~~~~~~~~~~~~~~0\leq x\leq 1,\\
&k_x(x,x)+k_y(x,x)+\frac{\text{d}}{\text{d}x}(k(x,x))= c +{\sigma},~0\leq x\leq 1,
\end{align}
\end{subequations}
with ${\sigma}$ being an arbitrary constant. Note that the existence of $k$ is guaranteed by \cite[Lemma 2.2]{Liu:2003}. Moreover, $k$ is twice continuously differentiable in $0 \leq y \leq x \leq 1$.}

As an application of Theorem~\ref{Theorem ISS}, we present below an ISS estimate in the spatial norm {for the system~\eqref{unstable PDE0} in closed loop with in-domain and boundary disturbances $f,d_0,d_1$, i.e.}:
\begin{subequations}\label{unstable PDE}
\begin{align}
u_t(x,t)=u_{xx}(x,t)+c u(x,t)
+f(x,t),\
&~(x,t)\in (0,1)\times\mathbb{R}_+,\\
u(0,t)=d_0(t),
u(1,t)=d_1(t)+U(t),\ &~t\in\mathbb{R}_+,\\
u(x,0)=u^0(x),\ &~x\in (0,1).
\end{align}
\end{subequations}

We have the following result.
\begin{proposition}\label{ISS-unstable} Assume that $c\in \mathbb{R}_+$ , $f\in C([0,1])$, $d_0,d_1 \in C (\mathbb{R}_{\geq 0})$,  {$u^0\in W^{1,2}(\Omega)$}. Assume further that $f$ is Lipschitz continuous w.r.t. its second variable. Let $k$ be the solution of \eqref{PDE-k} with some constant ${\sigma}\in \mathbb{R}_+$. Then, under the boundary feedback law  \eqref{+24}, the system~\eqref{unstable PDE} is EISS, having the estimate:
\begin{align*}
\sup_{x\in (0,1)}|u(x,T)|
 \leq &(1+M) \bigg((1+M)\sup_{x\in(0,1)}|u^0(x)|e^{-{\sigma} T}
    +{C}  \sup_{(x,t)\in (0,1)\times(0,T)}|f(x,t)|\notag\\&+ \sup_{t\in(0,T) }|d_0(t)|+ \sup_{t\in(0,T) }|d_1(t)|\bigg), \forall T>0,
\end{align*}
where {$C:= \min\left\{\frac{8\sqrt{2}}{\pi} , \frac{2\sqrt{2}}{\min\{1,c\}}\right\} $} and $M:= \sum_{i=0}^{+\infty}\frac{(c+\sigma)^{i+1}4^i}{(i!)^2}<+\infty$.
\end{proposition}

\begin{proof} Let $
w(x,t):=u(x,t)+\int_0^x k(x,y)u(y,t)\text{d}y
$, which transforms \eqref{unstable PDE} to (see \cite[Lemma 2.4]{Liu:2003}):
\begin{align*}
&w_t=w_{xx}-{\sigma} w
+f,~~~~~~~~~~~~~~~~~~~~~~~~~~~~~~~~~~~~~~~~~~~~~
~~~(x,t)\in (0,1)\times\mathbb{R}_+,\\
&w(0,t)=d_0(t),w(1,t)=d_1(t),~~~~~~~~~~~~~~~~~~~~~~~~~~~~~~t\in\mathbb{R}_+,\\
&w(x,0)=u^0(x)+\int_0^xk(x,y)u^0(y)\text{d}y:=w^0(x),~~x\in (0,1).
\end{align*}
The inverse of the mapping $u\mapsto w$ is given by
\begin{align}\label{12171}
u(x,t)=w(x,t)+\int_0^x l(x,y)w(y,t)\text{d}y,
\end{align}
 where $l$ is twice continuously differentiable in $0 \leq y \leq x \leq  1$ and satisfies
\begin{align*}
&l_{xx}(x,y)-l_{yy}(x,y)=(-{\sigma}-c)l(x,y),~~~~~~~~0\leq y\leq x\leq 1,\\
&l(x,0)=0,~~~~~~~~~~~~~~~~~~~~~~~~~~~~~~~~~~~~~~~~~~~~~~~~~~~~~~~0\leq x\leq 1,\\
&l_x(x,x)+l_y(x,x)+\frac{\text{d}}{\text{d}x}(l(x,x))=-{\sigma}-c,~0\leq x\leq 1.
\end{align*}
Note that  for $M:= \sum_{i=0}^{+\infty}\frac{(c+{\sigma})^{i+1}4^i}{(i!)^2}$, which is convergent, $k$ and $l$  satisfy (see the proof of \cite[Lemma 2.2]{Liu:2003}):
\begin{align*}
\max_{0 \leq y \leq x \leq  1} |k(x,y)|\leq M,
\max_{0 \leq y \leq x \leq  1} |l(x,y)|\leq M, \forall T>0,
\end{align*}
which along with the definition of $w^0$ and \eqref{12171} imply
\begin{subequations}\label{12173}
\begin{align}
\sup_{x\in (0,1)}|u(x,T)|\leq&(1+M) \sup_{x\in (0,1)}|w(x,T)|, \forall T>0,\\
\sup_{x\in (0,1)}|w^0(x)|\leq&(1+M) \sup_{x\in (0,1)}|u^0(x)|.
\end{align}
\end{subequations}
By Theorem~\ref{Theorem ISS}(ii), it follows that
\begin{align}\label{1217}
\sup_{x\in (0,1)}|w(x,T)|
 \leq &\sup_{x\in(0,1)}|w^0(x)|e^{-{\sigma} T}
      +{C_0}2^{\frac{3q-4}{2q-4}} \sup_{(x,t)\in (0,1)\times(0,T)}|f(x,t)| +\sup_{t\in(0,T) }|d_0(t)|+\sup_{t\in(0,T) }|d_1(t)|,
\end{align}
where $q$ and $C_{0}$ are {constants specified} in Theorem~\ref{Theorem ISS}(ii).

By \eqref{12173} and \eqref{1217}, we obtain
\begin{align}\label{121700}
\sup_{x\in (0,1)}|u(x,T)|
 \leq &(1+M) \bigg((1+M)\sup_{x\in(0,1)}|u^0(x)|e^{-{\sigma} T}
       + {C_0}2^{\frac{3q-4}{2q-4}} \sup_{(x,t)\in (0,1)\times(0,T)}|f(x,t)|\notag\\&+ \sup_{t\in(0,T) }|d_0(t)|+ \sup_{t\in(0,T) }|d_1(t)|\bigg).
\end{align}

 On one hand, the Agmon's inequality \cite[Lemma 2.4]{Krstic:2008} and a variation of  {the} Wirtinger's inequality \cite[Remark 2.2]{Krstic:2008} give
\begin{align}
 \|v\|_{L^q(0,1)}\leq &  \sup_{x\in (0,1)}|v(x)|\notag\\
 \leq&  \sqrt{2\|v \|_{L^2(0,1)}\|v_x\|_{L^2(0,1)}}\label{012805}\\
 \leq&  \frac{2}{\sqrt{\pi} }\|v_x\|_{L^2(0,1)}  ,\forall v\in W^{1,2}_0(0,1),\forall q\in (2,+\infty).\notag
\end{align}
Thus, $C_P$ in Lemma~\ref{Sobolev inequality}(ii) can be chosen as $C_P = \frac{2}{\sqrt{\pi} }$.

 {On the other hand, it follows from \eqref{012805} that for any $v\in W^{1,2}_0(0,1)$ and all $q\in (2,+\infty)$:
 \begin{align*}
 \|v\|_{L^q(0,1)}
 \leq   \sqrt{2\|v \|_{L^2(0,1)}\|v_x\|_{L^2(0,1)}}
 \leq   \frac{1}{\sqrt{2} } (\|v \|_{L^2(0,1)}+\|v_x\|_{L^2(0,1)}).
\end{align*}
 Thus, $C_S$ in Lemma~\ref{Sobolev inequality}(ii) can be chosen as $C_S = \frac{1}{\sqrt{2} }$.}

 {Finally, putting $C_0: =\min\left\{\frac{2C_{S}^2}{\min\{1,\underline{c}\}}, C_{P}^2 \right\}=\min\left\{ \frac{4}{\pi }, \frac{1}{\min\{1,c\}}\right\} $} into \eqref{121700} and letting $q\rightarrow +\infty$, we obtain the desired result.$\hfill\blacksquare$
\end{proof}
\begin{remark} It is worth noting that under the boundary feedback law  \eqref{+24} and with additional compatibility conditions on the initial and boundary data, EISS in $L^\infty$-norm was  established  for the classical solution of system  \eqref{unstable PDE} by using a method combined with transforming, De~Girogi iteration, splitting equations, and Lyapunov arguments in \cite{Zheng:201702}. While in this paper, the EISS in the spatial sup-norm is established for the weak solution without any compatibility condition, and only the technique of transforming and De~Girogi iteration are used in the proof.
\end{remark}
\begin{remark}\label{remark: unstable}
Since it is challenge to apply the backstepping method to $1$-D super-linear problems, or to PDEs defined on higher dimensional {spatial domains}, how to establish {the} ISS in $L^\infty$-norm w.r.t. boundary disturbances for a super-linear parabolic PDE having {a destabilizing} term, or a PDE defined on an arbitrary dimensional domain is still an open problem.
\end{remark}

\section{Application of RKES to nonlinear parabolic cascade  systems}\label{Sec. cascade systems}
In this section, as an application of the main results presented in Section~\ref{Sec: main results}, we show how to apply RKES to establish stability estimates in the sup-norm, or the spatial sup-norm, for a class of {nonlinear parabolic systems in cascade} connected either on the boundary or in the domain. Specifically, for a fixed integer $k\geq 2$ and $j\in \{1,2,...,k\}$, given functions {$a_j$, $c_j$, $m_j$, $h_j$, $\phi_j$, $f$, $d$,} we consider  the following systems coupled on the boundary:
\begin{equation*}\label{example}
(\Sigma_j)~\left\{\begin{aligned}
{\mathscr{L}_j}[u_j]+{h_j}(x,t,u_j)
 =&0~~~~~~~~\text{in}\ \Omega\times \mathbb{R}_+,\\
{a_j}\frac{\partial u_j}{\partial\bm{\nu}}+m_ju_j=&d_j~~~~~~\text{on}\ \partial \Omega \times \mathbb{R}_{+},\\
u_j(\cdot,0)=&{\phi_j} (\cdot)~~\text{in}\   \Omega,
\end{aligned}\right.
\end{equation*}
where {$\mathscr{L}_j[u]:=u_t-\text{div}  \ (a_j\nabla u) +c_ju$, and} for $(x,t)\in \partial \Omega\times\mathbb{R}_+$:
\begin{align}\label{open-loop}
d_1(x,t):=d(x,t),d_j(x,t):=u_{j-1}(x,t),\forall j\in [2,k],
\end{align}
 or
\begin{align}\label{closed-loop}
d_1(x,t):= u_{k}(x,t),
d_j(x,t):=u_{j-1}(x,t), \forall j\in [2,k].
\end{align}
{We also} consider the following equations coupled over the domain:
\begin{equation*}\label{example'}
(\Sigma_j')~\left\{\begin{aligned}
{\mathscr{L}_j}[u_j]+h_j(x,t,u_j)
 =&f_j~~~~~~~\text{in}\ \Omega\times \mathbb{R}_+,\\
u_j=&d_j~~~~~~~\text{on}\ \partial \Omega \times \mathbb{R}_{+},\\
u_j(\cdot,0)=&\phi_j (\cdot)~~~\text{in}\   \Omega,
\end{aligned}\right.
\end{equation*}
where for $(x,t)\in  \Omega\times\mathbb{R}_+$:
\begin{align}\label{D-open-loop}
f_1(x,t):=f(x,t),f_j(x,t):=u_{j-1}(x,t),\forall j\in [2,k],
\end{align}
 or
\begin{align}\label{D-closed-loop}
 f_1(x,t):=  u_{k}(x,t),
 f_j(x,t):= u_{j-1}(x,t), \forall j\in [2,k].
\end{align}

In this section, for any $T>0$, we intend to establish respectively:
  \begin{enumerate}[(i)]
  \item  the estimate  in the sup-norm, i.e., $\sup_{(x,t)\in Q_T}|u_j(x,t)|$; and
      \item the estimate in the spatial sup-norm, i.e., $\sup_{x\in \Omega}|u_j(x,T)|$,
  \end{enumerate}
for the  considered cascade systems, where $u_j$ is the solution of the $j$-th subsystem.

{For $j\in[1,k]$, we always assume that $a_j\in  \mathbb{A}$, $c_j\in \mathbb{C}$, $m_j \in  \mathbb{M}$, $h_j\in \mathbb{H}$, $\phi_j \in \mathbb{U}$.  Let $\underline{a_j},\underline{m_j} $ and $\underline{c_j}$ be defined by \eqref{underline}.}  Let \begin{align*}
 {\Phi_j}:=&\max\left\{\sup_{x\in\Omega}|\phi_1(x)|, ...,\sup_{x\in\Omega}|\phi_j(x)|\right\},
m_0:=  \min \{ \underline{m_1} ,..., \underline{m_k}\}>0,
a_0:= {\frac{1}{|\Omega|^{\frac{q-2}{q}}2^{\frac{3q-4}{2q-4}}\min\{\tau_1,...,\tau_k\}}>0},
\end{align*}
where
$
\tau_j:=\left\{\begin{aligned}
 &\frac{C_{P}^2}{\underline{a_j}},~~~~~~~~~~~~~~~~~~~~~~~~~~~~\text{if}~\underline{c_j}=0\\
 &\min\left\{\frac{2C_{S}^2}{\min\{\underline{a_j},\underline{c_j}\}},\frac{C_{P}^2}{\underline{a_j}}\right\},~\text{if}~\underline{c_j}>0
\end{aligned}\right.
$,
 and  the constants $q,C_S,C_{P}$ are specified in Lemma~\ref{Sobolev inequality}. We have the following two propositions. 
\begin{proposition} \label{stability-open-loop} {For the system $(\Sigma_j)$,  assume that $d\in\mathbb{D}$. Moreover, assume that  the structural conditions  \eqref{4} and \eqref{increasing property} are satisfied with $(h,f)=(h_j,0)$, and $h=h_j$, respectively.} Then the following statements hold true:
\begin{enumerate}[(i)]
 \item For $j\in [1,k]$, if $u_j$ is the solution of the system $(\Sigma_j)$ with the Robin boundary condition  given by \eqref{open-loop}, then
\begin{align}\label{26'}
\sup_{(x,t)\in Q_T} |u_j(x,t)|
 \leq &\frac{m_0}{m_0-1}\left(1-\frac{1}{m_0^j}\right) {\Phi_j}
     +\frac{1}{m_0^j}\sup_{(x,t)\in\partial_l Q_T}|d(x,t)|,\forall T>0.
\end{align}
%

 Furthermore, if $\underline{c_j}>0$, then
\begin{align}\label{26''}
\sup_{x\in \Omega} |u_j(x,T)|
 \leq &\frac{m_0}{m_0-1}\left(1-\frac{1}{m_0^j}\right) {\Phi_j}e^{-\underline{c_j}T}
      +\frac{1}{m_0^j}\sup_{(x,t)\in\partial_l Q_T}|d(x,t)|,\forall T>0.
\end{align}
 \item Assume further that $m_0>1$. For $j\in[1,k]$, if $u_j$ is the solution  of the system  $(\Sigma_j)$ with the Robin boundary condition  given by \eqref{closed-loop}, then
\begin{align}\label{121201}
\sup_{(x,t)\in Q_T} |u_j(x,t)|
 \leq &\frac{m_0}{m_0-1} {\Phi_k},\forall T>0.
     \end{align}

     Furthermore, if $\underline{c_j}>0$, then
     \begin{align}\label{121202}
\sup_{x\in \Omega} |u_j(x,T)|
 \leq &\frac{m_0}{m_0-1} {\Phi_k} e^{-\underline{c_j}T},\forall T>0.
     \end{align}
     \end{enumerate}
\end{proposition}

\begin{proposition} \label{D-stability-open-loop}
  {For the system $(\Sigma_j')$,   assume that $f\in\mathbb{F}$, $d_j\in \mathbb{D}$. Moreover, assume that  the structural condition   \eqref{4} is satisfied with $(h,f)=(h_j,f)$ for $j=1$, and $(h,f)=(h_j,0)$ for $j\in [2,k]$. Meanwhile, assume that  the structural condition \eqref{increasing property} is satisfied with  $h=h_j$ for all $j\in [1,k]$.}  Then the following statements hold true:
\begin{enumerate}[(i)]
\item
For $j\in [1,k]$, if $u_j$ is the solution of the system $(\Sigma_j')$ with \eqref{D-open-loop}, then
\begin{align}\label{-26'}
 \sup_{(x,t)\in Q_T} |u_j(x,t)|
 \leq  \frac{a_0}{a_0-1}\left(1-\frac{1}{a_0^j}\right) \Phi_j
   +\frac{1}{a_0^j}\sup_{(x,t)\in Q_T}|f(x,t)| +\sum_{i=1}^{j}\frac{1}{a_0^{j-i}}\sup_{(x,t)\in\partial_l Q_T}|d_i(x,t)|,\forall T>0.
\end{align}
Furthermore, if $\underline{c_j}>0$, then
\begin{align}
 \sup_{x\in \Omega} |u_j(x,T)|
 \leq  \frac{a_0}{a_0-1}\left(1-\frac{1}{a_0^j}\right)\Phi_j e^{-\underline{c_j}T}
 +\frac{1}{a_0^j}\sup_{(x,t)\in Q_T}|f(x,t)| +\sum_{i=1}^{j}\frac{1}{a_0^{j-i}}\sup_{(x,t)\in\partial_l Q_T}|d_i(x,t)|,\forall T>0.\label{A34}
\end{align}
\item Assume further that $a_0>1$. For $j\in [1,k]$, if $u_j$ is the solution of the system $(\Sigma_j')$ with \eqref{D-closed-loop}, then
\begin{align}
 \sup_{(x,t)\in Q_T} |u_j(x,t)|
 \leq &\frac{a_0}{a_0-1} \Phi_k +\frac{a_0^k}{a_0^k-1} \sum_{i=1}^{k}\frac{1}{a_0^{k-i}}\sup_{(x,t)\in\partial_l Q_T}|d_i(x,t)|,\forall T>0.\label{A35}
     \end{align}
     Furthermore, if $\underline{c_j}>0$, then
     \begin{align}\label{A36}
 \sup_{x\in \Omega} |u_j(x,T)|
 \leq  \frac{a_0}{a_0-1}\Phi_k e^{-\underline{c_j}T}  +\frac{a_0^k}{a_0^k-1} \sum_{i=1}^{k}\frac{1}{a_0^{k-i}}\sup_{(x,t)\in\partial_l Q_T}|d_i(x,t)|,\forall T>0.
     \end{align}
   \end{enumerate}
\end{proposition}
\begin{remark}
{It should be mentioned that in \cite{Dashkovskiy:2013}, the ISS and a small-gain theorem were established for a class of interconnected systems, provided that  ISS-Lyapunov functions  of the subsystems are known and a small-gain condition holds. As an application of the obtained results, small-gain conditions for guaranteeing  the 0-UGAS in the spatial $L^2$-norm  were proposed for a class of linear, and nonlinear, interconnected parabolic PDEs with homogeneous Dirichlet boundary conditions, respectively. For  general interconnected reaction-diffusion systems, it is reasonable to believe that such small-gain conditions  depend on the coefficients of the reaction and diffusion terms; see \cite{Dashkovskiy:2013} for two special cases. For the interconnected system $(\Sigma_j')$ coupled via \eqref{closed-loop}, the small-gain condition is characterized by $a_0>1$. While, for the system $(\Sigma_j)$ coupled on the boundary given in \eqref{D-closed-loop}, the small-gain condition is  characterized solely by $m_0>1$. Moreover, the small-gain conditions proposed in this paper can be used for guaranteeing not only the 0-UGAS, but also the ISS, in the spatial sup-norm, for the considered systems with either Robin or Dirichlet boundary conditions.}
\end{remark}

\paragraph*{ Proof of Proposition~\ref{stability-open-loop}}
The proof is based on  using {RKES} repeatedly and composed of 4 steps.

\emph{Step 1: proof of \eqref{26'}.}  let $v_j$ be the solution of the following system:
\begin{align*}
\mathscr{L}_j[v_j]+h_j(x,t,v_j)
 =&0~~~~~~~~\text{in}\ \Omega\times \mathbb{R}_+,\\
a_j\frac{\partial v}{\partial\bm{\nu}}+m_jv_j=&0~~~~~~~~\text{on}\ \partial \Omega \times \mathbb{R}_{+},\\
v_j(\cdot,0)=&\phi_j (\cdot)~~\text{in}\   \Omega.
\end{align*}
The maximum estimate of $v_j$ is given by (see \eqref{22} in Appendix)
\begin{align}\label{29'}
\sup_{(x,t)\in Q_T} |v_j(x,t)| \leq  \sup_{x\in\Omega}|\phi_j(x)|\leq {\Phi_j}, \forall T>0.
\end{align}

For $T>0$, we deduce from Theorem~\ref{Theorem RKES-Robin}(i) and \eqref{29'} that
\begin{align}
 \sup_{(x,t)\in Q_T} |u_j(x,t)|   \leq& \sup_{(x,t)\in Q_T} |v_j(x,t)|+\sup_{(x,t)\in Q_T} |u_j(x,t)-v_j(x,t)|  \notag\\
\leq
&{\Phi_j}+\frac{1}{\underline{m_j}}\sup_{(x,t)\in\partial_l Q_T}|d_j(x,t)-0|
\notag\\ =& {\Phi_j}+\frac{1}{\underline{m_j}}\sup_{(x,t)\in\partial_l Q_T}|u_{j-1}(x,t)|
\notag\\
\leq&  {\Phi_j}+\frac{1}{\underline{m_j}}\sup_{(x,t)\in Q_T}|u_{j-1}(x,t)|\notag\\
\leq& {\Phi_j}+\frac{1}{\underline{m_j}}\bigg( \sup_{(x,t)\in Q_T} |v_{j-1}(x,t)| + \sup_{(x,t)\in Q_T}|u_{j-1}(x,t)-v_{j-1}(x,t)|\bigg)\notag\\
\leq& {\Phi_j}+\frac{1}{\underline{m_j}}\left({\Phi_j}+ \frac{1}{\underline{m_{j-1}}}\sup_{(x,t)\in\partial_l Q_T}|d_{j-1}(x,t)|\right)
\notag\\
=  & {\Phi_j}\left(1+\frac{1}{\underline{m_j}}\right)+\frac{1}{\underline{m_j}\cdot\underline{m_{j-1}}}\sup_{(x,t)\in\partial_l Q_T}|u_{j-2}(x,t)|\notag\\
\leq& {\Phi_j}\left(1+\frac{1}{\underline{m_j}}\right)+\frac{1}{\underline{m_j}\cdot\underline{m_{j-1}}}\sup_{(x,t)\in Q_T}|u_{j-2}(x,t)|
\notag\\
\leq &~\cdots~\notag\\ \leq & {\Phi_j}\left(1+\frac{1}{\underline{m_j}}+\frac{1}{\underline{m_j}\cdot\underline{m_{j-1}}}+\cdots+\frac{1}{\underline{m_j}
\cdots\underline{m-3}}\right)
+\frac{1}{\underline{m_j}\cdots\underline{m_{2}}}\sup_{(x,t)\in Q_T}|u_{1}(x,t)| \label{26}\\
\leq &{\Phi_j}\left(1+\frac{1}{\underline{m_j}}+\frac{1}{\underline{m_j}\cdot\underline{m_{j-1}}}+\cdots+\frac{1}{\underline{m_j}\cdots\underline{m_{2}}}\right)
 +\frac{1}{\underline{m_j}\cdots\underline{m_1}}\sup_{(x,t)\in\partial_l Q_T}|d(x,t)| \notag\\
\leq &{\Phi_j}\left(1+\frac{1}{m_0}+\frac{1}{m_0^{2}}+\cdots+\frac{1}{m_0^{j-1}}\right)
 +\frac{1}{m_0^j}\sup_{(x,t)\in\partial_l Q_T}|d(x,t)|\notag\\
=&\frac{m_0}{m_0-1}\left(1-\frac{1}{m_0^j}\right){\Phi_j}
     +\frac{1}{m_0^j}\sup_{(x,t)\in\partial_l Q_T}|d(x,t)|,\notag
\end{align}
which gives \eqref{26'}.

\emph{Step 2: proof of \eqref{26''}.}  For any constant ${\delta}\in (0,\underline{c}_j) $, let
$w_j:=u_je^{\delta t},\widetilde{c}_j:=c_j-\delta,\widetilde{h}_j(x,t,w_j):=h_j(x,t,w_je^{\delta t})e^{\delta t},\widetilde{d}_j(x,t):= e^{\delta t}d_j(x,t)$, and $\widehat{\mathscr{L}}_j[w_j]:=(w_j)_t-\text{div}  \ (a_j\nabla w_j) +\widetilde{c}_jw_j$. By direct computations, we have
\begin{equation*}\label{examplev2}
(\widetilde{\Sigma}_j)~\left\{\begin{aligned}
\widehat{\mathscr{L}}_j[w_j]+\widetilde{h}_j(x,t,w_j)
 =&0~~~~~~~~~\text{in}\ \Omega\times \mathbb{R}_+,\\
a_j\frac{\partial w_j}{\partial\bm{\nu}}+m_jw_j=&\widetilde{d}_j~~~~~~~\text{on}\ \partial \Omega \times \mathbb{R}_{+},\\
w_j(\cdot,0)=&\phi_j(\cdot)~~~\text{in}\   \Omega.
\end{aligned}\right.
\end{equation*}
Note that $\min_{x\in \overline{\Omega}}\widetilde{c}_j=\underline{c_j}-{\delta}>0$. Then, applying \eqref{26'} to the system $(\widetilde{\Sigma}_j)$, we obtain
\begin{align*}
\sup_{(x,t)\in Q_T} |w_j(x,t)|
 \leq &\frac{m_0}{m_0-1}\left(1-\frac{1}{m_0^j}\right) {\Phi_j}
      +\frac{1}{m_0^j}\sup_{(x,t)\in\partial_l Q_T}|\widetilde{d}_1(x,t)|,\forall T>0,
\end{align*}
which implies
\begin{align*}
\sup_{x\in \Omega} |u_j(x,T)|
 \leq &\frac{m_0}{m_0-1}\left(1-\frac{1}{m_0^j}\right){\Phi_j}e^{-\delta T}
    +\frac{1}{m_0^j}\sup_{(x,t)\in\partial_l Q_T}|d(x,t)|,\forall T>0.
\end{align*}
Letting $\delta\rightarrow \underline{c_j}$, we obtain \eqref{26''}.

\emph{Step 3: proof of \eqref{121201}}. Without loss of generality, we consider the system $(\Sigma_j)$ with \eqref{closed-loop} for $j=k$. We deduce from \eqref{26}, \eqref{closed-loop}, and Theorem~\ref{Theorem RKES-Robin}(i) that
\begin{align}
 \sup_{(x,t)\in Q_T} |u_k(x,t)|   \leq & {\Phi_k}\left(1+\frac{1}{m_0}+\cdots+\frac{1}{m_0^{k-2}}\right) +\frac{1}{m_0^{k-1}}\sup_{(x,t)\in Q_T}|u_{1}(x,t)| \notag\\
\leq &{\Phi_k}\left(1+\frac{1}{m_0}+\cdots+\frac{1}{m_0^{k-2}}\right) +\frac{1}{m_0^{k-1}}\left({\Phi_k}+\frac{1}{m_0 }\sup_{(x,t)\in\partial_l Q_T}|u_k(x,t)|\right)\notag\\
\leq
&\frac{m_0}{m_0-1}\left(1-\frac{1}{m_0^{k}}\right){\Phi_k}+\frac{1}{m_0^{k}}\sup_{(x,t)\in Q_T}|u_{k}(x,t)|,\label{B1}
\end{align}
which along with $m_0>1$ implies \eqref{121201}.

\emph{Step 4: proof of \eqref{121202}.} Using transformation as in Step 2, considering $(\widetilde{\Sigma}_j)$ with the boundary conditions given by \eqref{closed-loop}, and applying \eqref{121201}, we   get \eqref{121202}. $\hfill\blacksquare$
 \paragraph*{Proof of Proposition~\ref{D-stability-open-loop}} Indeed, for $j\in [1,k]$, let $v_j$ be the solution of the following system:
\begin{align*}
\mathscr{L}_j[v_j]+h_j(x,t,v_j)
 =&0~~~~~~~~\text{in}\ \Omega\times \mathbb{R}_+,\\
v_j=&0~~~~~~~~\text{on}\ \partial \Omega \times \mathbb{R}_{+},\\
v_j(\cdot,0)=&\phi_j (\cdot)~~\text{in}\   \Omega,
\end{align*}
Analogous to \eqref{29'}, the maximum estimate of $v_j$ is given by
\begin{align}\label{-29'}
\sup_{(x,t)\in Q_T} |v_j(x,t)| \leq  \sup_{x\in\Omega}|\phi_j(x)|, \forall T>0.
\end{align}

  First, we prove \eqref{-26'}. For $T>0$, we deduce from Theorem~\ref{Theorem RKES-Robin}(ii) and \eqref{-29'}
that
\begin{align}
 \sup_{(x,t)\in Q_T} |u_j(x,t)|  \leq& \sup_{(x,t)\in Q_T} |v_j(x,t)|+\sup_{(x,t)\in Q_T} |u_j(x,t)-v_j(x,t)| \notag\\
\leq
& {\Phi_j}+\frac{1}{a_0}\sup_{(x,t)\in  Q_T}|f_j(x,t)-0|+\sup_{(x,t)\in\partial_l Q_T}|d_j(x,t)-0|
\notag\\
\leq&   {\Phi_j}+\frac{1}{a_0}\sup_{(x,t)\in Q_T}|u_{j-1}(x,t)|+\sup_{(x,t)\in\partial_l Q_T}|d_j(x,t)| \notag\\
\leq&  {\Phi_j}+\frac{1}{a_0}\bigg( \sup_{(x,t)\in Q_T} |v_{j-1}(x,t)| + \sup_{(x,t)\in Q_T}|u_{j-1}(x,t)-v_{j-1}(x,t)|\bigg)+\sup_{(x,t)\in\partial_l Q_T}|d_j(x,t)|\notag\\
\leq&  {\Phi_j}+\frac{1}{a_0}\bigg( {\Phi_j}+ \frac{1}{a_0}\sup_{(x,t)\in Q_T}|f_{j-1}(x,t)| +\sup_{(x,t)\in\partial_l Q_T}|d_{j-1}(x,t)|\bigg)
+\sup_{(x,t)\in\partial_l Q_T}|d_j(x,t)|
\notag\\
\leq &~\cdots~\notag\\ \leq & \frac{a_0}{a_0-1}\left(1-\frac{1}{a_0^j}\right)  {\Phi_j}
   +\frac{1}{a_0^j}\sup_{(x,t)\in Q_T}|f(x,t)| +\sum_{i=1}^{j}\frac{1}{a_0^{j-i}}\sup_{(x,t)\in\partial_l Q_T}|d_i(x,t)|, \label{B}
\end{align}
which gives \eqref{-26'}.

{Now we prove \eqref{A35}. Without loss of generality, we consider the case of $j=k$. Indeed, analogous to \eqref{B} (see also \eqref{B1}), we have
\begin{align}
 \sup_{(x,t)\in Q_T} |u_k(x,t)|
\leq & \frac{a_0}{a_0-1}\left(1-\frac{1}{a_0^k}\right)  {\Phi_k}
   +\frac{1}{a_0^k}\sup_{(x,t)\in Q_T}|f_1(x,t)| +\sum_{i=1}^{k}\frac{1}{a_0^{k-i}}\sup_{(x,t)\in\partial_l Q_T}|d_i(x,t)|, \notag\\
   =  & \frac{a_0}{a_0-1}\left(1-\frac{1}{a_0^k}\right)  {\Phi_k}+\sum_{i=1}^{k}\frac{1}{a_0^{k-i}}\sup_{(x,t)\in\partial_l Q_T}|d_i(x,t)|
    +\frac{1}{a_0^k}\sup_{(x,t)\in Q_T}|u_k(x,t)|,\notag
\end{align}
which along with $a_0>1$ gives \eqref{A35}.}

 Finally, using the technique of transforming as in the proof of Proposition~\ref{stability-open-loop}, we obtain \eqref{A34} and \eqref{A36}. $\hfill\blacksquare$
%

%
\section{Concluding remarks}\label{Sec. remarks}
This paper proposed a new {method for establishing the} ISS in the spatial sup-norm for nonlinear parabolic PDEs with boundary and in-domain disturbances. More precisely, we introduced the notion of \emph{RKES} to describe the uniform dependence of solutions on the external disturbances. Based on RKES in  the (spatial and time) sup-norm, we proved the ISS in the spatial sup-norm for a class of higher dimensional nonlinear PDEs with Dirichlet and Robin boundary disturbances, respectively. Two examples were provided to illustate the obtained ISS results. In addition, as an application of the introduced notion of \emph{RKES}, we also established stability estimates in the sup-norm and spatial sup-norm for a class of {parabolic systems in cascade} coupled over the domain and on the boundary of the domain, respectively.

It should be mentioned that the approach presented in this paper is well suited for ISS analysis of weak solutions to higher dimensional nonlinear PDEs with Dirichlet or Robin boundary conditions. However, it seems to be difficult to apply the proposed method to obtain the ISS in the spatial sup-norm for PDEs with Neumann boundary disturbances due to the usage of De~Giorgi iteration. Therefore, there is a need to overcome this obstacle and establish ISS estimates in the spatial sup-norm for a wider class of PDEs with various boundary disturbances by a unified approach, which will be considered in our future work.

Besides, as indicated in Remark~\ref{remark: unstable}, it is also necessary to develop new methods or tools to address the ISS in the spatial sup-norm for $1$-D super-linear parabolic PDEs with destabilizing terms, or parabolic PDEs  defined over a higher dimensional domain and having destabilizing terms.

\appendix
\section{Proofs of main results}\label{appendix}


We present some basic Sobolev embedding inequalities that will be used in the proofs of stabilities.
\begin{lemma} \label{Sobolev inequality}(Theorem 1.3.2 and  1.3.4 of \cite{Wu2006})
Let $\Omega$ be a bounded, open
subset of $\mathbb{R}^n (n\geq 1)$, and suppose that $\partial \Omega$ is $C^1$. For $n=1,2$ and $q\in (2,+\infty)$, or $n\geq 3$ and $q\in (2,\frac{2n}{n-2})$, the  following inequalities hold true:
\begin{enumerate}
\item[(i)] $
\|v\|_{L^q(\Omega)}\leq C_S (\|v\|_{L^2(\Omega)}+\|\nabla v\|_{L^2(\Omega)}),\forall v\in W^{1,2}(\Omega)
 $,
 \item[(ii)] $
\|v\|_{L^q(\Omega)}\leq C_{P}  \|\nabla v\|_{L^2(\Omega)},\forall v\in W_0^{1,2}(\Omega),
 $
\end{enumerate}
where $C_S$ and $C_{P}$ are positive constants depending only on $q,n$, and $\Omega$.
\end{lemma}

\paragraph*{Proof of Theorem~\ref{Theorem RKES-Robin}}
We prove first Theorem~\ref{Theorem RKES-Robin}(i). Let $u_i $ be the solution of the system $\Sigma (\mathbb{U},\mathbb{F},\mathbb{D})$ corresponding to the data $ (u^0,f_i,d_i)\in \mathbb{U}\times\mathbb{F}\times\mathbb{D}, i=1,2 $.

Consider $w=u_1-u_2$, which satisfies:
\begin{subequations}\label{Mixed'}
\begin{align}
\mathscr{L}[w]+h(x,t,u_1)-h(x,t,u_2)=&\widetilde{f}~~\text{in}\ \Omega\times \mathbb{R}_+,\\
a\frac{\partial w}{\partial\bm{\nu}}+mw
=&\widetilde{d}~~\text{on}\  \partial \Omega \times \mathbb{R}_{+},\label{Mixed'b}\\
w(\cdot,0)=&0~~\text{in}\   \Omega,
\end{align}
\end{subequations}
where $\widetilde{d}:=d_1-d_2,\widetilde{f}:=f_1-f_2$.

We proceed by De Giorgi iteration. Specifically, for any $T>0$, let {}{$k_0:=\max\Big\{0,\frac{1}{\underline{m}}\sup_{\partial_l Q_T}\widetilde{d}\Big\}$}. For $k\geq k_0$ and $0< t_1<t_2< T$, let $ \eta(x,t):=(w(x,t)-k)_+\chi_{[t_1,t_2]}(
t)$, where {$s_+:=\max\{s,0\}$ for $s\in \mathbb{R}$}, and $\chi_{[t_1,t_2]}(t) $ is the character function on $[t_1,t_2]$. By virtue of  Proposition~\ref{existence}, and that $W^{1,p}(\Omega)\hookrightarrow W^{1,2}(\Omega)\hookrightarrow L^2 (\Omega)\hookrightarrow (W^{1,2}(\Omega))'\hookrightarrow (W^{1,p}(\Omega))'$  for $p\geq 2$, we have $\eta\in L^\infty((0,T);(W^{1,p}(\Omega))')$ with $\eta_t\in L^\infty((0,T);L^{p'}(\Omega))$.
Then, $\eta$ can be chosen as a test function for \eqref{Mixed'}.

By the Fubini's theorem and integrating by parts, we have
\begin{align*}
  -\int_{0}^T\int_{\Omega}w\eta_t\text{d}x\text{d}t
 = -\int_\Omega w(x,T)\eta(x,T)\text{d}x+\int_\Omega w(x,0)\eta(x,0)\text{d}x + \int_{0}^T\int_{\Omega}w_t\eta\text{d}x\text{d}t
 = \int_{0}^T\int_{\Omega}w_t\eta\text{d}x\text{d}t.
\end{align*}
It follows that
\begin{align}\label{+1603}
&\int_{0}^T\int_{\Omega}(w-k)_t(w-k)_+\chi_{[t_1,t_2]}(t)\text{d}x\text{d}t -\int_{0}^T\int_{\partial \Omega}(\widetilde{d}-mw)(w-k)_+\chi_{[t_1,t_2]}(t)\text{d}S\text{d}t\notag\\
&+\int_{0}^T\int_{\Omega}a|\nabla (w-k)_+ |^2\chi_{[t_1,t_2]}(t)\text{d}x\text{d}t
+\int_{0}^T\int_{\Omega}cw(w-k)_+\chi_{[t_1,t_2]}(t) \text{d}x\text{d}t \notag\\
&+\int_{0}^T\int_{\Omega}(h(x,t,u_1)-h(x,t,u_2))(w-k)_+\chi_{[t_1,t_2]}(t) \text{d}x\text{d}t\notag\\
=&\int_{0}^T\int_{\Omega}\widetilde{f}(w-k)_+\chi_{[t_1,t_2]}(t) \text{d}x\text{d}t.
\end{align}
Note that for $w\geq k\geq k_0\geq 0$, it follows that  $-mw\leq -mk_0\leq -\underline{m}k_0\leq -\sup_{\partial_l Q_T} \widetilde{d}$, which implies that
\begin{align}
 \int_{0}^T\int_{\partial \Omega}(\widetilde{d}-mw)(w-k)_+\chi_{[t_1,t_2]}(t)\text{d}S\text{d}t
 \leq  \int_{0}^T\int_{\partial \Omega}(\widetilde{d}-\sup_{\partial_l Q_T} \widetilde{d})(w-k)_+\chi_{[t_1,t_2]}(t)\text{d}S\text{d}t
 \leq  0.\label{9-1}
\end{align}
In addition, for $w\geq k\geq k_0\geq 0$, it follows that $u_1=u_2+w\geq u_2$,  which and \eqref{increasing property} give
\begin{align}
\int_{0}^T\int_{\Omega}(h(x,t,u_1)-h(x,t,u_2))(w-k)_+\chi_{[t_1,t_2]}(t) \text{d}x\text{d}t\geq 0.\label{9-2}
\end{align}
It is obvious that
\begin{align}
 \int_{0}^T\int_{\Omega}cw(w-k)_+\chi_{[t_1,t_2]}(t) \text{d}x\text{d}t   =&\int_{0}^T\int_{\Omega}c((w-k)_+)^2\chi_{[t_1,t_2]}(t) \text{d}x\text{d}t +\int_{0}^T\int_{\Omega}kc(w-k)_+\chi_{[t_1,t_2]}(t) \text{d}x\text{d}s\notag\\
 \geq& \underline{c}\int_{0}^T\int_{\Omega}((w-k)_+)^2\chi_{[t_1,t_2]}(t) \text{d}x\text{d}t .\label{9-3}
\end{align}
Then, by \eqref{+1603}, \eqref{9-1}, \eqref{9-2}, and \eqref{9-3}, we obtain
\begin{align*}
&\int_{0}^T\int_{\Omega}(w-k)_t(w-k)_+\chi_{[t_1,t_2]}(t)\text{d}x\text{d}t
 +\underline{a}\int_{0}^T\int_{\Omega}|\nabla (w-k)_+ |^2\chi_{[t_1,t_2]}(t)\text{d}x\text{d}t
\notag\\&+\underline{c}\int_{0}^T\int_{\Omega}((w-k)_+)^2\chi_{[t_1,t_2]}(t) \text{d}x\text{d}t\notag\\
\leq& \int_{0}^T\int_{\Omega}\widetilde{f}(w-k)_+\chi_{[t_1,t_2]}(t) \text{d}x\text{d}t.
\end{align*}
Hence
\begin{align*}
 \frac{1}{2}\int_{t_1}^{t_2}\frac{\text{d}}{\text{d}t}\int_{\Omega}((w-k)_+)^2\text{d}x\text{d}t
+\underline{a}\int_{t_1}^{t_2}\int_{\Omega}|\nabla (w-k)_+ |^2\text{d}x\text{d}t
 +\underline{c}\int_{t_1}^{t_2}\int_{\Omega}((w-k)_+)^2\text{d}x\text{d}t
 \leq  \int_{t_1}^{t_2}\int_{\Omega}\widetilde{f}(w-k)_+\text{d}x\text{d}t,
\end{align*}
i.e.,
\begin{align*}
 \frac{1}{2} (I_k(t_2)-I_k(t_1))
+\underline{a}\int_{t_1}^{t_2}\int_{\Omega}|\nabla (w-k)_+ |^2\text{d}x\text{d}t
 +\underline{c}\int_{t_1}^{t_2}\int_{\Omega}((w-k)_+)^2\text{d}x\text{d}t
  \leq \int_{t_1}^{t_2}\int_{\Omega}\widetilde{f}(w-k)_+\text{d}x\text{d}t,
\end{align*}
where $I_k(t):=\int_\Omega((w(x,t)-k)_+)^2\text{d}x$.

Suppose that $I_k(t_0)=\max_{t\in[0,T]}I_k(t)$ with some $t_0\in [0,T]$.
Due to $I_k(0)=0$ and $I_k(t)\geq 0 $, we can assume that $t_0\in (0,T]$ without loss of generality.

If $t_0=T$, then $I'_k(T)\geq 0$. Thus $I'_k(t)\geq 0$ on $(T-\delta,T]$ for some $\delta>0$. Then, there exists a sufficiently small constant $\varepsilon>0$ such that $I_k(T-\varepsilon)-I_k(T-2\varepsilon)\geq 0$. Taking $t_2=T-\varepsilon$ and $t_1=T-2\varepsilon>0$, we obtain
\begin{align*}
 \frac{\underline{a}}{\varepsilon}\int_{T-2\varepsilon}^{T-\varepsilon}\int_{\Omega}|\nabla (w-k)_+ |^2\text{d}x\text{d}t
 +\frac{\underline{c}}{\varepsilon}\int_{T-2\varepsilon}^{T-\varepsilon}\int_{\Omega}((w-k)_+)^2\text{d}x\text{d}t
  \leq    \frac{1}{\varepsilon}\int_{T-2\varepsilon}^{T-\varepsilon}\int_{\Omega}|\widetilde{f}|(w-k)_+\text{d}x\text{d}t.
\end{align*}

Letting $ \varepsilon\rightarrow 0^+$, we get for such  $t_0:=T$:
\begin{align}\label{1204-13}
 \underline{a}\int_{\Omega}|\nabla (w(x,t_0)-k)_+ |^2\text{d}x
+\underline{c}\int_{\Omega}((w(x,t_0)-k)_+)^2\text{d}x
\leq \int_{\Omega}|f(x,t_0)|(w(x,t_0)-k)_+\text{d}x.
\end{align}

If $t_0\in (0,T)$, we can take $t_2=t_0$ and $t_1=t_0-\varepsilon>0$ for a small $\varepsilon>0$. Analogously, we can obtain \eqref{1204-13}. Thus, \eqref{1204-13} holds true whenever $t_0\in (0,T]$.

Using Lemma~\ref{Sobolev inequality}(i), we have
\begin{align}\label{1204-14}
 \|(w(x,t_0)-k)_+\|^2_{L^q(\Omega)}  \leq   2C_{S}^2\Big(\|(w(x,t_0)-k)_+\|^2_{L^2(\Omega)}
 +\|\nabla (w(x,t_0)-k)_+\|^2_{L^2(\Omega)}\Big),
\end{align}
where $q$ and $C_S$ are the same as in Lemma~\ref{Sobolev inequality}(i).

Let $A_{k}(t):=\{x\in \Omega;w(x,t)>k\}$. By \eqref{1204-13}, \eqref{1204-14}, $\underline{a}>0$, and $\underline{c}>0$, we have
\begin{align}\label{1204-15}
 \frac{\min\{\underline{a},\underline{c}\}}{2C_{S}^2}\left(\int_{A_{k}(t_0)}|w(x,t_0)-k|^{q}\text{d}x\right)^{\frac{2}{q}}
=&\frac{\min\{\underline{a},\underline{c}\}}{2C_{S}^2}\|(w(x,t_0)-k)_+\|^2_{L^q(\Omega)}
\notag\\
\leq&
\underline{a}  \|\nabla (w(x,t_0)-k)_+\|^2_{L^2(\Omega)} +\underline{c}\|(w(x,t_0)-k)_+\|^2_{L^2(\Omega)}
\notag\\
\leq&
\int_{\Omega}|\widetilde{f}(x,t_0)|(w(x,t_0)-k)_+\text{d}x.
\end{align}
By {the} H\"{o}lder's inequality, we have
\begin{align*}
 \int_{\Omega}|\widetilde{f}(x,t_0)|(w(x,t_0)-k)_+\text{d}x  \leq &\left( \int_{A_{k}(t_0)}|\widetilde{f}(x,t_0)|^{q'}\text{d}x\right)^{\frac{1}{q'}}\left(\int_{A_{k}(t_0)}|w(x,t_0)-k|^{q}\text{d}x\right)^{\frac{1}{q}},
 \end{align*}
which along with \eqref{1204-15} gives
\begin{align*}
  \left(\int_{A_{k}(t_0)}|w(x,t_0)-k|^{q}\text{d}x\right)^{\frac{1}{q}}  \leq &\frac{2C_{S}^2}{\min\{\underline{a},\underline{c}\}}\left( \int_{A_{k}(t_0)}|\widetilde{f}(x,t_0)|^{q'}\text{d}x\right)^{\frac{1}{q'}}\notag\\
 \leq &\frac{2C_{S}^2}{\min\{\underline{a},\underline{c}\}} \|\widetilde{f}\|_{L^\infty(Q_T)}|A_k(t_0)|^{\frac{1}{q'}}\notag\\
 \leq &\frac{2C_{S}^2}{\min\{\underline{a},\underline{c}\}} \|\widetilde{f}\|_{L^\infty(Q_T)}\mu_{k}^{\frac{1}{q'}},
 \end{align*}
where $q':=\frac{q}{q-1}$, $\mu_{k}:=\sup_{t\in(0,T)}|A_{k}(t)|$, and $|A_{k}(t_0)|$ denotes the $n$-dimensional Lebesgue measure of $A_{k}(t_0)$.
Then, we may proceed exactly as in the proof of \cite[Theorem 4.2.1]{Wu2006} to obtain
\begin{align*}
w
\leq & k_0+\frac{2C_{S}^2}{\min\{\underline{a},\underline{c}\}}|\Omega|^{\frac{q-2}{q}}2^{\frac{3q-4}{2q-4}}
     \|\widetilde{f}\|_{L^\infty(Q_T)}\ \ \text{a.e.\ in\ }Q_T,
\end{align*}
which along with the continuities of $w$ and $f$ yields
\begin{align}
w
\leq & \max\left\{0,\frac{1}{\underline{m}}\sup_{\partial_l Q_T}\widetilde{d}\right\}
     +\frac{2C_{S}^2}{\min\{\underline{a},\underline{c}\}}|\Omega|^{\frac{q-2}{q}}2^{\frac{3q-4}{2q-4}} \sup_{ Q_T}|\widetilde{f}|\ \ \text{in\ }Q_T.\label{1204-16}
\end{align}

We need  to prove the lower boundedness of $w$. Indeed, it suffices to set $\overline{w}:=-w=u_2-u_1$, and consider the equation
\begin{align*}
\mathscr{L}[\overline{w}]+h(x,t,u_2)-h(x,t,u_1)=&-\widetilde{f}~~ \text{in}\ \Omega\times \mathbb{R}_+,\\
a\frac{\partial \overline{w}}{\partial\bm{\nu}}+m\overline{w}
=&-\widetilde{d}~~\text{on}\  \partial \Omega \times \mathbb{R}_{+},\\
\overline{w}(\cdot,0)=&0~~~~~~~\text{in}\   \Omega.
\end{align*}
 Let {}{$\overline{k}_0=\max\Big\{0,\frac{1}{\underline{m}}\sup_{\partial_l Q_T}(-\widetilde{d})\Big\}$}. Proceeding as above, we obtain
 \begin{align}
-w=  \overline{w}
\leq   \overline{k}_0
     +\frac{2C_{S}^2}{\min\{\underline{a},\underline{c}\}}|\Omega|^{\frac{q-2}{q}}2^{\frac{3q-4}{2q-4}} \sup_{ Q_T}|\widetilde{f}|\ \ \text{in\ }Q_T.\label{1204-17}
\end{align}
Finally, by \eqref{1204-16} and \eqref{1204-17}, we have
\begin{align*}
 \sup_{Q_T}|w|\leq \frac{1}{\underline{m}}\sup_{\partial_l Q_T}|\widetilde{d}|
     +\frac{2C_{S}^2}{\min\{\underline{a},\underline{c}\}}|\Omega|^{\frac{q-2}{q}}2^{\frac{3q-4}{2q-4}} \sup_{ Q_T}|\widetilde{f}|,
\end{align*}
which gives the   result stated in Theorem~\ref{Theorem RKES-Robin}(i).

Now we  prove  Theorem~\ref{Theorem RKES-Robin}(ii). Consider \eqref{Mixed'} by replacing \eqref{Mixed'b} with
\begin{align*}
  w=\widetilde{d}~\text{on}\  \partial \Omega \times \mathbb{R}_{+}.
\end{align*}
For any $T>0$, let {}{$k_0:=\max\Big\{0,\sup_{\partial_l Q_T}\widetilde{d}\Big\}$}. For $k\geq k_0$ and $0< t_1<t_2< T$, let $ \eta(x,t):=(w(x,t)-k)_+\chi_{[t_1,t_2]}(
t)$. It suffices to apply De Giorgi iteration as in the proof of Theorem~\ref{Theorem RKES-Robin} (i).

Indeed, if $\underline{c}\geq 0$, \eqref{1204-13} can be reduced to
\begin{align*}
\underline{a}\int_{\Omega}|\nabla (w(x,t_0)-k)_+ |^2\text{d}x
\leq \int_{\Omega}|f(x,t_0)|(w(x,t_0)-k)_+\text{d}x.
\end{align*}
By Lemma~\ref{Sobolev inequality}(ii), \eqref{1204-14} can be reduced to
\begin{align*}
\|(w(x,t_0)-k)_+\|^2_{L^q(\Omega)}\leq C_{P}^2\|\nabla (w(x,t_0)-k)_+\|^2_{L^2(\Omega)},
 \end{align*}
 where $q$ and $C_{P}$ are the same as in Lemma~\ref{Sobolev inequality}(ii). Hence \eqref{1204-15} becomes
 \begin{align}
 \frac{\underline{a}}{ C_{P}^2 }\left(\int_{A_{k}(t_0)}|w(x,t_0)-k|^{q}\text{d}x\right)^{\frac{2}{q}}
\leq
\int_{\Omega}|\widetilde{f}(x,t_0)|(w(x,t_0)-k)_+\text{d}x.\label{012804}
\end{align}
Then, analogous to \eqref{1204-16}, we obtain the following estimate:
\begin{align*}
w
\leq & \max\left\{0,\sup_{\partial_l Q_T}\widetilde{d}\right\}
     +\frac{ C_{P}^2}{\underline{a}}|\Omega|^{\frac{q-2}{q}}2^{\frac{3q-4}{2q-4}} \sup_{ Q_T}|\widetilde{f}|\ \ \text{in\ }Q_T.
\end{align*}
The lower boundedness of $w$ can be estimated in the similar way, and the boundedness of $w$ specified in \eqref{012801} is guaranteed.

 {Now for $\underline{c}>0$, we shall determine an appropriate coefficient of $\left(\int_{A_{k}(t_0)}|w(x,t_0)-k|^{q}\text{d}x\right)^{\frac{2}{q}}$. Indeed, for $v\in W^{1,2}_0(\Omega)$, we also have $v\in W^{1,2}(\Omega)$. Thus \eqref{012804} and \eqref{1204-15} hold true at the same time. Then we obtain
 \begin{align*}
  \left(\int_{A_{k}(t_0)}|w(x,t_0)-k|^{q}\text{d}x\right)^{\frac{2}{q}}
\leq
C_0\int_{\Omega}|\widetilde{f}(x,t_0)|(w(x,t_0)-k)_+\text{d}x,
\end{align*}
where $C_0:=  \min\left\{\frac{2C_{S}^2}{\min\{\underline{a},\underline{c}\}},\frac{C_{P}^2}{\underline{a}}\right\}$.
 Finally, \eqref{012802} is guaranteed.}

$\hfill\blacksquare$

\paragraph*{Proof of Theorem~\ref{Theorem ISS}}
We only prove Theorem~\ref{Theorem ISS}(i), since the proof of Theorem~\ref{Theorem ISS}(ii) can be proceeded in the same way.

We first prove that {the} system~\eqref{PDE} with the Robin boundary condition \eqref{Robin} is 0-UGAS w.r.t. the state in the spatial sup-norm. Indeed, let $v$ be the solution of the following equation:
\begin{subequations}\label{Mixed-0ks}
\begin{align}
\mathscr{L}[v]+h(x,t,v)
 =&0~~~~~~~~\text{in}\ \Omega\times \mathbb{R}_+,\\
a\frac{\partial v}{\partial\bm{\nu}}+mv
=&0~~~~~~~~\text{on}\ \partial \Omega \times \mathbb{R}_{+},\\
v(\cdot,0)=&u^0(\cdot)~~\text{in}\   \Omega.
\end{align}
\end{subequations}
 For any constant $\delta\in (0,\underline{c}) $, let
$w:=ve^{\delta t},\widetilde{c}:=c-\delta,\widetilde{h}(x,t,w):=h(x,t,we^{\delta t})e^{\delta t},\widetilde{f}(x,t):=\widetilde{d}(x,t):=0$. By direct computations, we have
\begin{align*}
w_t-\text{div}  \ (a\nabla w) +\widetilde{c}w+\widetilde{h}(x,t,w)
 =&\widetilde{f}~~~~~~~~\text{in}\ \Omega\times \mathbb{R}_+,\\
a\frac{\partial w}{\partial\bm{\nu}}+mw
=&\widetilde{d}~~~~~~~~\text{on}\ \partial \Omega \times \mathbb{R}_{+},\\
w(\cdot,0)=&u^0(\cdot)~~\text{in}\   \Omega.
\end{align*}
Note that  $\min_{x\in\overline{\Omega}}\widetilde{c}=\underline{c}-\delta >0 $. Then one may
apply  De Giorgi iteration as in the proof of Theorem~\ref{Theorem RKES-Robin}(i), and obtain
\begin{align}\label{22}
 \sup_{(x,t)\in Q_T} |w(x,t)|  \leq  k_0:=\max\bigg\{\sup_{x\in\Omega}|u^0(x)|,\sup_{(x,t)\in Q_T}|\widetilde{f}(x,t)|, \sup_{(x,t)\in \partial_l Q_T}|\widetilde{d}(x,t)|\bigg\}  = \sup_{x\in\Omega}|u^0(x)|, \forall T>0,
\end{align}
which along with the continuity of $w$ in $t=T$ implies that
\begin{align*}
\sup_{x\in \Omega} |w(x,T)| \leq\sup_{(x,t)\in Q_T} |w(x,t)| \leq \sup_{x\in\Omega}|u^0(x)|.
\end{align*}
It follows that
\begin{align*}
\sup_{x\in \Omega} |v(x,T)| \leq   e^{-\delta T} \sup_{x\in\Omega}|u^0(x)|.
\end{align*}
Letting $\delta\rightarrow \underline{c}$, we have
\begin{align}\label{+201806+02'}
\sup_{x\in \Omega} |v(x,T)| \leq   e^{-\underline{c}  T} \sup_{x\in\Omega}|u^0(x)|.
\end{align}
Finally, by \eqref{190625}, \eqref{+201806+02'}, and Proposition~\ref{Proposition UKC}, we conclude that {the} system~\eqref{PDE} with the Robin boundary condition \eqref{Robin} is EISS in the spatial sup-norm w.r.t. in-domain and boundary disturbance $(f,d)$ in $\mathbb{F}\times \mathbb{D}$, having the estimate \eqref{++13''}. $\hfill\blacksquare$

\end{document}